\documentclass[11pt]{amsart}
\usepackage{amsbsy}
\usepackage{amsmath}
\usepackage[alphabetic, backrefs, msc-links]{amsrefs}
\usepackage{amssymb}
\usepackage{amsthm}
\usepackage{amsxtra}
\usepackage{cancel}
\usepackage{comment}
\usepackage{fullpage}
\usepackage{graphicx}
\usepackage{hyperref}
\usepackage{paralist}
\usepackage{mathrsfs}
\usepackage{stmaryrd}
\usepackage{textcomp}
\usepackage[all,2cell]{xy}
\usepackage{enumitem}
\usepackage{pxfonts}
\usepackage{fancyhdr}
\usepackage{tikz}
\usepackage{tikz-qtree}
\usetikzlibrary{cd,matrix,arrows,decorations.pathmorphing}
\usepackage{bbm}
\makeatletter
\pgfqkeys{/tikz/cs}{
  latitude/.store in=\tikz@cs@latitude,
  longitude/.style={angle={#1}},
  theta/.style={latitude={#1}},
  rho/.style={angle={#1}}
}
\tikzdeclarecoordinatesystem{xyz spherical}{
  \pgfqkeys{/tikz/cs}{angle=0,radius=0,latitude=0,#1}%
  \pgfpointspherical{\tikz@cs@angle}{\tikz@cs@latitude}{\tikz@cs@xradius}
}
\makeatother

\tikzset{my color/.code=\pgfmathparse{(#1+90)/180*100}\pgfkeysalso{every path/.style={color=red!\pgfmathresult!blue}}}

\DeclareRobustCommand{\SkipTocEntry}[5]{}

\newtheorem{prop}{Proposition}[section]

\newtheorem{thm}[prop]{Theorem}
\newtheorem{cor}[prop]{Corollary}
\newtheorem{lem}[prop]{Lemma}

\theoremstyle{definition}
\newtheorem{defn}[prop]{Definition}

\newtheorem{cons}[prop]{Construction}

\newtheorem{notn}[prop]{Notation}

\theoremstyle{remark}

\newtheorem{rem}[prop]{Remark}
\newtheorem{warn}[prop]{Warning}

\numberwithin{equation}{prop}
\DeclareMathOperator{\Ab}{Ab}

\DeclareMathOperator{\Cart}{Cart}
\DeclareMathOperator{\Cat}{Cat}

\DeclareMathOperator{\Ch}{Ch}

\DeclareMathOperator{\DK}{DK}

\DeclareMathOperator{\Ext}{Ext}

\DeclareMathOperator{\Fun}{Fun}

\DeclareMathOperator{\holim}{holim}

\DeclareMathOperator{\Hom}{Hom}
\DeclareMathOperator{\HomL}{Hom^{\mathrm L}}

\DeclareMathOperator{\id}{id}

\DeclareMathOperator{\Kan}{Kan}

\DeclareMathOperator{\Map}{Map}

\DeclareMathOperator{\Perf}{Perf}

\DeclareMathOperator{\QC}{QC}
\DeclareMathOperator{\QCoh}{QCoh}

\DeclareMathOperator{\res}{res}

\DeclareMathOperator{\RHom}{\mathbf RHom}

\DeclareMathOperator{\Sch}{Sch}

\DeclareMathOperator{\Set}{Set}

\DeclareMathOperator{\sk}{sk}

\DeclareMathOperator{\SMap}{\mathscr Map}

\DeclareMathOperator{\SRHom}{\mathbf R\mathscr Hom}
\DeclareMathOperator{\St}{St}

\DeclareMathOperator{\Tw}{Tw}

\DeclareMathOperator{\Un}{Un}
\newcommand{\natmap}[5]{\xymatrix{
#1 \ar@/^/[r]^{#3} \ar@/_/[r]_{#4} \ar@{}[r]|{\Downarrow_{#5}} & #2
}}
\newcommand{\adjoints}[4]{\xymatrix{#1 \ar@/^/[rr]^{#3} & \perp & #2 \ar@/^/[ll]^{#4}}}
\newcommand{\smalladjoints}[4]{\xymatrix@C=5pt{#1 \ar@/^/[rr]^{#3} & \perp & #2 \ar@/^/[ll]^{#4}}}
\newcommand{\xtworightarrows}[4]{\xymatrix{#1 \ar@<.5ex>[r]^{#3} \ar@<-.5ex>[r]_{#4} & #2}}
\newcommand{\xtwoleftarrows}[4]{\xymatrix{#1 &  \ar@<-.5ex>[l]_{#3} \ar@<.5ex>[l]^{#4}  #2}}
\newcommand{\xrightleftarrows}[4]{\xymatrix{#1 \ar@/^/[r]^{#3} & #2 \ar@/^/[l]^{#4}}}

\newcommand{\dg}[0]{\mathrm{dg}}

\newcommand{\op}[0]{\mathrm{op}}

\newcommand{\pr}[0]{\mathrm{pr}}

\usepackage{marginnote}

\theoremstyle{remark}

\newtheorem{remark}[prop]{Remark}
\theoremstyle{definition}
\newtheorem{definition}[prop]{Definition}
\theoremstyle{plain}
\newtheorem{proposition}[prop]{Proposition}
\newtheorem{corollary}[prop]{Corollary}

\newtheorem{lemma}[prop]{Lemma}

\title{On the $\infty$-stack of complexes over a scheme}

\author{A. Dhillon}
\email{adhill3@uwo.ca}
\address{Western University, Canada}

\author{P. Zs\'amboki}
\email{pzsambok@uwo.ca}
\address{Western University, Canada}

\newcommand{\bC}{{\textbf C}}

\newcommand{\bA}{{\textbf A}}

\newcommand{\sF}{{\mathcal F}}

\newcommand{\ndg}{N_{\rm dg}}

\newcommand{\msD}{{\mathscr D}}
\newcommand{\msX}{{\mathscr X}}
\newcommand{\msC}{{\mathscr C}}

\DeclareMathOperator{\st}{St}
\DeclareMathOperator{\str}{str}

\begin{document}

\begin{abstract}
{We study fppf descent for  enhanced derived categories. We revisit the work of
	\cite{hs} and \cite{toen2008homotopical} in a lax context. More precisely, we construct a
    Cartesian and coCartesian fibration ${}^\op\msD^+_S\rightarrow N(\Sch_S)$ whose fibre over an $S$-scheme $T$ is the opposite $\mathscr D^+(T)^\op$ of the quasi-category of bounded below complexes of $\mathscr O_T$-modules. We show that this fibration satisfies fppf-descent
    for schemes.
    The main components in the proof are limit formulas for the mapping spaces in the section quasi-category $\Gamma(K,\mathscr X)$ and its subcategory of Cartesian sections $\Gamma_{\Cart}(K,\mathscr X)$ of a Cartesian fibration over a quasi-category $\mathscr X\to K$. These formulas are of independent interest.
    Since our construction gives a functor of quasi-categories of complexes, it yields $\SRHom$ $\infty$-stacks with natural composition maps.
     The final section gives an explicit description of the $\infty$-group structure of the automorphism $\infty$-group of a complex.}
\end{abstract}

\maketitle

\tableofcontents

\section{Introduction}

In this article we study gluing of complexes up to quasi-isomorphism over fppf covers
of schemes. In trying to do so, we find that homotopies and higher homotopies between
quasi-isomorphisms naturally enter into the picture and we find ourselves in the world
of higher categories.

Descent for quasi-coherent sheaves has its origins in \cite{SGA1}. The stack of quasi-coherent
sheaves was first constructed here. It is important to note that we do not naturally have a functor of
quasi-coherent sheaves, but rather a psuedo-functor. In other words, for a pair of composable morphisms
of schemes
\[
U\xrightarrow{f} T\xrightarrow{g} S,
\]
the functors $(g\circ f)^*$ and $f^*\circ g^*$ are not the same but are canonically isomorphic.
The canonical isomorphisms satisfy certain other compatibilities. For this  reason it is easier to
state descent in terms of fibred categories $\QCoh\rightarrow \Sch$.

The first descent statements in derived categories
appeared in \cite{SGA4}*{Exp. Vbis}. In this article, given a fppf morphism $f:T\rightarrow S$
one expands it to its Cech nerve $T^\bullet/S$. One can then
consider the Grothendieck category of abelian sheaves on the nerve, and form its
derived category. A descent statement for sheaves and cohomology to $S$ can be proved in
this context. Note that one could consider the diagram of
derived categories associated to the Cech nerve. This category, lacks enough information
to prove descent and is markedly different from the category considered in
\cite{SGA4}. It is this gap, by suitably enhancing the derived category, that this
paper seeks to address. Furthermore, it is not shown in loc. cit. that the collection
of derived categories forms a stack.

This question was taken up
in \cite{hs}. This work introduces the notion of a Quillen presheaf. A strictification theorem is proved in the context of Quillen preseaves. The final section of this paper proves
descent in a very general setting for derived categories. These methods
were applied and extended in \cite{toen2008homotopical}.

These results described above consider various enhanced derived categories.
What if we work with ordinary derived categories? A gluable complex is a
complex whose negative self extensions vanish. These negative extension groups
are precisely the homotopy groups of the mapping spaces in the stack that we
consider below, see \ref{prop:Map in D}. Descent for universally gluable
complexes was first proved in \cite{BBD}, building upon the theory
of cohomological descent in \cite{SGA4}. Using this result, Lieblich
has shown that universally gluable complexes form an algebraic stack, see
\cite{lieblich2006moduli}.

In this article we will construct a fibration of enhanced derived categories in the form of a relative dg-nerve, see
\ref{s:fibration}. We will prove some basic properties of it, such as,
it is a presentable fibration (\ref{prop:D is Cartesian}), and it satisfies descent (\ref{thm:descent for complexes}).
In a nutshell we carry out the constructions of \cite{SGA1} for suitably enhanced derived
categories.

It has been long known that the category perfect complexes is a well behaved
generalisation of the category of vector bundles, see for instance \cite{TT}.
A first step towards studying moduli of perfect complexes is to define and prove descent
for this category, in other words show that it forms a (higher) stack.
To do so, we need to construct a sheaf of perfect complexes. The difficulty, just as in the case of quasi-coherent sheaves, is that naturally we can only construct a pseudofunctor as we will now explain. Let $T$ be an $S$-scheme. Then the category of perfect complexes on $T$ is a homotopical category, that is it can be equipped with a notion of weak equivalence, the quasi-isomorphisms. Then one can take the simplicial localization $\Perf(T)$. Let $V\xrightarrow hU\xrightarrow gT$ be morphisms of $S$-schemes. We then have derived pullback functors giving a diagram of simplicial sets
\begin{center}

\begin{tikzpicture}[xscale=4,yscale=1.5]
\node (C') at (0.5,1) {$\Perf(V)$};
\node (C) at (0,0) {$\Perf(U)$};
\node (D) at (1,0) {$\Perf(T)$.};
\path[->,font=\scriptsize,>=angle 90]
(C') edge node [left] {$Lh^*$} (C)
(D) edge node [right] {$Lg^*$} (C')
(D) edge node [above] {$L(hg)^*$} (C);
\end{tikzpicture}

\end{center}
But this diagram is only commutative up to a homotopy $(Lh^*)(Lg^*)\to L(hg)^*$. Therefore, to get an actual presheaf of simplicial sets
\[
\Sch_S\xrightarrow{\Perf}\Set_\Delta,
\]
one needs to take a strictification. This is a strict functor, which is only equivalent to the natural lax functor. Usually this construction is only made for affine $S$, from which the stack of perfect complexes on a general scheme is obtained by taking the homotopy limit along an affine open cover. This makes the structure even more inexplicit.

We get a natural explicit description using the quasi-categorical generalization of the Grothendieck construction, which is recalled in \S\ref{ss:Cartesian}. On fibres, we can take the dg-nerve, the simplices of which are by construction diagrams in the dg-category of complexes. This is recalled in \S\ref{ss:dgNerve}. Then using functorial dg-flat resolutions, we get derived pullback maps with a functorial choice of homotopies making the diagrams as above commutative (\S\ref{ss:cotorsion pairs}). Homotopy limits and the statement of descent
in a higher categorical setting is recalled in \S\ref{ss:descent}. Descent for a Cartesian fibration over the nerve of a site can be formulated using the quasi-category of Cartesian sections. In \S\ref{s:twisted}, we describe mapping spaces in section quasi-categories and their subcategories on Cartesian sections as homotopy limits of
certain diagrams. These formulas can be thought of as an extension of the description in
\cite{gepner2017lax}. This section is of independent interest and it is likely that
these homotopy limit descriptions will have other applications. The main object of study is introduced in \S\ref{s:fibration}. We construct a presentable fibration
$$
{}^\op\mathscr D_S\xrightarrow p N(\Sch)_S
$$
with simplices explicitly given as diagrams in the dg-category of complexes (\S\ref{s:fibration}). We can then formulate and check fppf descent in this setting (\S\ref{s:theorem}). The main theorem is proved in \ref{thm:descent for complexes}. The theorem has various applications,
for example we obtain a description of $QC(T)$ for an ordinary scheme. This object plays the
role of the category of quasi-coherent sheaves in derived algebraic geometry, see
\cite{ben-zvi2010integral}. The final section gives an explicit description of the $\infty$-group structure of the automorphism $\infty$-group of a complex. This topic will be expanded upon in future work.

\begin{notn}

In this paper, every scheme will be assumed to be quasi-compact and quasi-separated. Let's fix once and for all a morphism of schemes $X\xrightarrow fS$.

We will work with quasi-categories and largely follow the notation of \cite{lurie2009higher}. By an
$\infty$-category we will mean a quasi-category. See \ref{ss:quasicategories} for a brief introduction.
By an ordinary category we will mean a small quasi-category that is equivalent
to the nerve of an ordinary category, that is a set of objects and morphisms subject to the
usual conditions.

We might suppress the nerve of a 1-category from notation. That is, for example $\Sch_S$ as a simplicial set is $N(\Sch_S)$.

The $\infty$-category  of spaces will be denoted by $\mathscr S$. This is just the simplicial nerve of the simplicial category of Kan complexes.
\end{notn}


\section{Background on quasi-categories}\label{s:background}

\subsection{Quasi-categories as $\infty$-categories}\label{ss:quasicategories}

For a more detailed exposition, we refer the reader to \cite{lurie2009higher}*{\S1}. A quasi-category $X$ is a type of simplicial set which can model an $\infty$-category. Its vertices $\Delta^0\to X$ are objects, and its edges $\Delta^1\to X$ are 1-morphisms. The 2-simplices
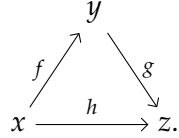
\begin{center}

\begin{tikzpicture}[xscale=2,yscale=1.5]
\node (C') at (0.5,1) {$y$};
\node (C) at (0,0) {$x$};
\node (D) at (1,0) {$z$.};
\path[->,font=\scriptsize,>=angle 90]
(C) edge node [left] {$f$} (C')
(C') edge node [right] {$g$} (D)
(C) edge node [above] {$h$} (D);
\end{tikzpicture}

\end{center}
we can to think of as composition diagrams. By abuse of notation, we will write $gf=h$ if $X$ has a 2-simplex as above. Note that this means that we can have $gf=h_1$ and $gf=h_2$ with $h_1\ne h_2$, that is composition is not unique. But one can check that if $X$ is a quasi-category, then we get $h_1\simeq h_2$ in this case.

For example, in the main example of the quasi-category $\mathscr D(T)$ which is recalled in \S\ref{ss:dgNerve}, a 2-simplex $\Delta^2\to\mathscr D(T)$ is given by

\begin{enumerate}

\item 3 complexes of injective $\mathscr O_T$-modules $I_0,I_1,I_2$,

\item 3 morphisms of complexes $I_0\xrightarrow{f_{01}}I_1,\,I_1\xrightarrow{f_{12}}I_2,\,I_0\xrightarrow{f_{02}}I_2$,

\item and a homotopy $I_0\xrightarrow{f_{012}}I_2[-1]$ such that $df_{012}=f_{12}f_{01}-f_{02}$.

\end{enumerate}

Note that an inner horn $\Lambda^2_1\to X$ given by $f_{01}$ and $f_{12}$ can be completed to a 2-simplex in many ways.

For a simplicial set $X$ to be a quasi-category, we need to be able to composes morphisms, and higher morphisms also. The relative notion is that of an inner fibration. A map of simplicial sets $X\xrightarrow pS$ is an \emph{inner fibration}, if it satisfies the right lifting property with respect to all inner horn inclusions $\Lambda^n_k\subset\Delta^n$ for $n\ge2$ and $0<k<n$. A simplicial set $X$ is a \emph{quasi-category}, if the canonical map $X\to\ast$ is an inner fibration. If $X$ is a quasi-category, then by $x\in X$ we mean that $x$ is a vertex: $x\in X_0$. Let $x\xrightarrow fy\xrightarrow g$ be a $\Lambda^2_1$-diagram in $X$. Then by the lifting property, it can be complexes to a 2-simplex. We will write $x\xrightarrow{gf}z$ for some composite we get this way.

As we said, composition is not unique in a quasi-category. But there are still ways to get mapping spaces and composition maps. Let $K$ and $L$ be simplicial sets. Then their \emph{join} $K\star L$ has as set of $n$-simplices
$$
(K\star L)_n=\bigsqcup_{-1\le i\le n}(K_{[0,i]}\times L_{[i+1,n]}),
$$
where we set $K_\emptyset=L_\emptyset=\ast$. Let $K$ be a simplicial set, $X$ a quasi-category, and $K\xrightarrow kX$ a diagram. Then the \emph{overcategory $X_{/p}$} has as set of $n$-simplices
$$
(X_{/p})_n=\{\Delta^n\star K\xrightarrow\sigma X:\sigma|K=k\}.
$$
Let $y\in X$ be an object and $\Delta^0\xrightarrow kX$ be its inclusion map. Then we denote $X_{/y}=X_{/k}$. Let $x\in X$ be another object. Then we define the \emph{right Hom space} $\Hom^R_X(x,y)=\{x\}\times_XX_{/y}$. Note that its set of $n$-simplices is
$$
\Hom^R_X(x,y)_n=\{\Delta^{n+1}\xrightarrow\sigma X:\sigma|\Delta^{[0,n]}=\{x\},\,\sigma|\Delta^{\{n+1\}}=\{y\}\}.
$$
One can check that $\Hom^R_X(x,y)$ is a Kan complex. It is one of the ways to define a mapping space in a quasi-category.

Let $\Delta^1\xrightarrow{y\xrightarrow fz}X$ be a morphism, and $x\in X$ another object. Then one can check that the restriction map $X_{/f}\to X_{/y}$ is a trivial fibration. Therefore, its pullback $\{x\}\times_XX_{/f}\xrightarrow r\{x\}\times_XX_{/y}$ is also a trivial fibration. Thus, it has a section $\{x\}\times_XX_{/y}\xrightarrow s\{x\}\times_XX_{/f}$. We can get a postcomposition by $f$ map as the composite
$$
f\circ\colon\Hom^R_X(x,y)=\{x\}\times_XX_{/y}\xrightarrow s\{x\}\times_XX_{/f}\to\{x\}\times_XX_{/z}=\Hom^R_X(x,z).
$$
Note that this map is not unique as it depends on the choice of the section $s$. Dually, we can define undercategories $X_{k/}$, left Hom spaces $\Hom^L_X(x,y)$, and precomposition by $f$ maps. There is a third version $\Hom_X(x,y)$ of the mapping space, with set of $n$-simplices
$$
\Hom_X(x,y)_n=\{\Delta^1\times\Delta^n\xrightarrow\sigma X:\sigma|\Delta^{\{0\}}\times\Delta^n=\{x\},\,\sigma|\Delta^{\{1\}}\times\Delta^n=\{y\}\}.
$$
One can show that the natural inclusions
$$
\Hom^L_X(x,y)\to\Hom_X(x,y)\leftarrow\Hom^R_X(x,y)
$$
are homotopy equivalences of Kan complexes \cite{lurie2009higher}*{Corollary 4.2.1.8}. Because of these equivalences, we write $\Map_X(x,y)$ to mean any Kan complex homotopy equivalent to any of these.

Let $\mathbf C$ be a category. Then its \emph{categorical nerve} $N(\mathbf C)$ is the simplicial set with set of $n$-simplices composable chains of morphisms of length $n$:
$$
N(\mathbf C)_n=\{c_0\xrightarrow{f_0}c_1\xrightarrow{f_1}\dotsb\xrightarrow{f_{n-1}}c_n\},
$$
the face maps are given by composition, and the degeneracy maps are given by identity maps. One can show that this gives a fully faithful functor
$$
\Cat\xrightarrow N\Set_\Delta.
$$
We denote its left adjoint by $\tau_1$. Let $X$ be a quasi-category. Then we can describe the category $\tau_1(X)$ as follows. The objects of $\tau_1(X)$ are the objects of $X$. Let $f,g\colon x\rightrightarrows y$ be two morphisms. Then we say that $f$ and $g$ are \emph{homotopic}, if there exists a 2-simplex in $X$ of the form
\begin{center}

\begin{tikzpicture}[xscale=2,yscale=1.5]
\node (C') at (0.5,1) {$y$};
\node (C) at (0,0) {$x$};
\node (D) at (1,0) {$y$.};
\path[->,font=\scriptsize,>=angle 90]
(C) edge node [left] {$f$} (C')
(C') edge node [right] {$\id_y$} (D)
(C) edge node [above] {$g$} (D);
\end{tikzpicture}

\end{center}
One can show that this relation is an equivalence, and that letting $\Hom_{\tau_1(X)}(x,y)$ be the set of morphisms modulo this relation, we can give a category structure to $\pi_1(X)$ using composition diagrams \cite{lurie2009higher}*{Proposition 1.2.3.8}.

Let $X\xrightarrow fY$ be a map of simplicial sets. Then $f$ is \emph{essentially surjective}, if $\tau_1(f)$ is an essentially surjective functor of categories. The map $f$ is \emph{fully faithful}, if for all $x,y\in X$, the induced map $\Map_X(x,y)\xrightarrow{f_{x,y}}\Map_Y(f(x),f(y))$ is a homotopy equivalence. The map $f$ is a \emph{categorical equivalence}, if it is both essentially surjective and fully faithful.

Using this notion, we can define the \emph{Joyal model structure} on the category $\Set_\Delta$. In it,

\begin{enumerate}

\item cofibrations are monomorphisms, and

\item weak equivalences are categorical equivalences.

\end{enumerate}

One can show that the fibrant objects in the Joyal model structure are precisely the quasi-categories \cite{lurie2009higher}*{Theorem 2.4.6.1}.

We have seen that we can get composition maps in a quasi-category, even if they are not unique. One can go further with strictification, and from a quasi-category $X$ get a simplicial category $\mathfrak C[X]$ with the same object set, and equivalent mapping spaces and composition maps. The statement uses the \emph{Bergner model structure} on the category $\Cat_\Delta$ on simplicial categories \cite{bergner2007model}*{Theorem 1.1}. In it,

\begin{enumerate}

\item Weak equivalences are \emph{DK-equivalences}, that is maps of simplicial categories $\mathbf C\xrightarrow F\mathbf D$ such that

\begin{enumerate}

\item for all $x,y\in\mathbf C$, the induced map $\Map_{\mathbf C}(x,y)\xrightarrow{F_{xy}}\Map_{\mathbf D}(Fx,Fy)$ is a weak equivalence, and

\item the induced map on the underlying categories $\pi_0\mathbf C\xrightarrow{\pi_0F}\pi_0\mathbf D$ is an equivalence of categories.

\end{enumerate}

\item Fibrations are \emph{local fibrations}, that is maps of simplicial categories $\mathbf C\xrightarrow F\mathbf D$ such that

\begin{enumerate}

\item for all $x,y\in\mathbf C$, the induced map $\Map_{\mathbf C}(x,y)\xrightarrow{F_{xy}}\Map_{\mathbf D}(Fx,Fy)$ is a Kan fibrations, and

\item for all $x\in\mathbf C$ and homotopy equivalence $Fx\xrightarrow{f}y'$ in $\mathbf D$ there exists a homotopy equivalence $x\xrightarrow fy$ in $\mathbf C$ such that $Ff=f'$.

\end{enumerate}

\end{enumerate}

Then there exists a Quillen equivalence $\smalladjoints{\Set_\Delta}{\Cat_\Delta}{\mathfrak C}{N_\Delta}$ \cite{lurie2009higher}*{Theorem 2.2.5.1}. See \cite{lurie2009higher}*{\S1.1.5} for the construction of $\mathfrak C$. Let $\mathbf C$ be a simplicial category. We call $N_\Delta(\mathbf C)$ its \emph{homotopy coherent nerve}. Let $X$ be a quasi-category. Then for all $x,y\in X$, we have a weak equivalence of simplicial sets $\Hom^R_X(x,y)\simeq\Map_{\mathfrak C[X]}(x,y)$ \cite{lurie2009higher}*{Corollary 2.2.2.10, Proposition 2.2.4.1}.


\subsection{Straightening-unstraightening between right fibrations and presheaves of spaces}

Let $\Kan\subseteq\Set_\Delta$ denote the full simplicial subcategory on Kan complexes. Then the \emph{quasi-category of spaces} $\mathscr S$ is its coherent nerve: $\mathscr S=N_\Delta\Kan$. Thus, a presheaf of Kan complexes on a quasi-category $\mathscr C$ can be given as a map of simplicial sets $\mathscr C^\op\to\mathscr S$. As this includes pseudofunctors, it is very difficult to define presheaves like this. Therefore, we employ the quasi-categorical generalization of the Grothendieck construction. The quasi-categorical generalization of the notion of a fibred category is the notion of a right fibration. A morphism of simplicial sets $X\to S$ is a \emph{right fibration}, if it satisfies the right lifting property with respect to the horn inclusions $\Lambda^n_k\subset\Delta^n$ for $n\ge1$ and $0<k\le n$.

Let $S$ be a simplicial set. The overcategory $(\Set_\Delta)_{/S}$ can be equipped with the \emph{contravariant model structure}. This is a left proper, combinatorial, simplicial model category \cite{lurie2009higher}*{Propositions 2.1.4.7 and 2.1.4.8} in which

\begin{enumerate}

\item A cofibration is a monomorphism.

\item An $S$-morphism of simplicial sets $X\to Y$ is a \emph{contravariant equivalence}, if the induced map
$$
S\sqcup_XX^\vartriangleright\to S\sqcup_YY^\vartriangleright
$$
is a categorical equivalence.

\end{enumerate}
One can show that every contravariant fibration is a right fibration, and moreover the fibrant objects of $(\Set_\Delta)_{/S}$ are precisely the right fibrations over $S$, see \cite{lurie2009higher}*{Proposition 2.1.4.9}.

The main result \cite{lurie2009higher}*{2.2.1.2} is that equipping $(\Set_\Delta)^{\mathfrak C[S]}$ with the projective model structure induced by the Quillen model structure, we obtain a Quillen equivalence
$$
\adjoints{(\Set_\Delta)_{/S}}{(\Set_\Delta)^{\mathfrak C[S^\op]}}{\St}{\Un}.
$$
The functors $\St$ and $\Un$ are called the \emph{straightening} and \emph{unstraightening} functors. Let $X\xrightarrow pS$ be a right fibration, and $S^\op\xrightarrow f\mathscr S$ a presheaf. Then by the $(\mathfrak C,N_\Delta)$-adjunction, $f$ corresponds to a simplicial functor $\mathfrak C[S^\op]\xrightarrow{f^\sharp}\Set_\Delta$. We say that \emph{$p$ is classified by $f$}, if $p\simeq\Un(f^\sharp)$.


\subsection{Marked simplicial sets, simplicial localization, and straightening-unstraightening}\label{ss:Cartesian}

Let $K$ be a simplicial set and $\mathscr C$ a quasi-category. Then the Hom simplicial set $\Fun(K,\mathscr C)$ with $n$-simplices given by
$$
\Fun(K,\mathscr C)_n=\Hom_{\Set_\Delta}(K\times\Delta^n,\mathscr C)
$$
is a quasi-category \cite{lurie2009higher}*{Proposition 1.2.7.3}.

The simplicial category $\Cat_\infty^\Delta$ is formed as follows.

\begin{enumerate}

\item Its objects are small quasi-categories.

\item For quasi-categories $\mathscr C,\mathscr D$, the mapping space $\Map_{\Cat_\infty^\Delta}(\mathscr C,\mathscr D)$ is the largest Kan complex in $\Fun(\mathscr C,\mathscr D)$.

\end{enumerate}

We let $\Cat_\infty$ be the coherent nerve $N_\Delta(\Cat_\infty^\Delta)$. We would like to get unstraightenings of presheaves of quasi-categories $\mathscr C^\op\to\Cat_\infty$. These will be the Cartesian fibrations. Let $X\xrightarrow pS$ be an inner fibration. Let $x\xrightarrow ey$ be an edge in $X$. We say that $e$ is a \emph{$p$-Cartesian edge}, if the canonical map
$$
X_{/e}\to X_{/y}\times_{S_{/p(y)}}S_{/p(e)}
$$
is a trivial fibration. This definition makes sense because of the following. As described in \S\ref{ss:quasicategories}, we can get postcomposition maps $\Hom^R_X(z,x)\xrightarrow{e\circ}\Hom^R_X(x,y)$ and $\Hom^R_S(p(z),p(x))\xrightarrow{p(e)\circ}\Hom^R_S(p(z),p(y))$. Then $e$ is $p$-Cartesian if and only if for all $z\in X$, the diagram
\begin{center}

\begin{tikzpicture}[xscale=4,yscale=1.5]
\node (C') at (0,1) {$\Hom^R_X(z,x)$};
\node (D') at (1,1) {$\Hom^R_X(z,y)$};
\node (C) at (0,0) {$\Hom^R_S(p(z),p(x))$};
\node (D) at (1,0) {$\Hom^R_S(p(z),p(y))$.};
\path[->,font=\scriptsize,>=angle 90]
(C') edge node [above] {$e\circ$} (D')
(C') edge node [right] {$p$} (C)
(D') edge node [right] {$p$} (D)
(C) edge node [above] {$p(e)\circ$} (D);
\end{tikzpicture}

\end{center}
is homotopy Cartesian \cite{lurie2009higher}*{Proposition 2.4.4.3}. Note that if $S=\ast$, then $e$ is a $p$-Cartesian edge if and only if it is an equivalence. The inner fibration $X\xrightarrow pS$ is a \emph{Cartesian fibration}, if for all vertices $y\in X$ and edges $\Bar x\xrightarrow{\Bar e}p(y)$ in $S$, there exists a $p$-Cartesian edge $e$ such that $p(e)=\Bar e$.

To get a straightening-unstraightening construction between Cartesian fibrations and presheaves in quasi-categories, we need to keep track of where morphisms take Cartesian edges. Therefore, it needs to be formulated using marked simplicial sets. A \emph{marked simplicial set} is a pair $(X,\mathscr E)$ where

\begin{enumerate}

\item $X$ is a simplicial set, and

\item $\mathscr E$ is a collection of edges of $X$ containing the degenerate edges.

\end{enumerate}

For a simplicial set $X$, we have the two extreme cases.

\begin{enumerate}

\item The marked edges in the marked simplicial set $X^\flat$ are only the degenerate edges.

\item In the marked simplicial set $X^\sharp$, every edge is marked.

\end{enumerate}

Let $\Set_\Delta^+$ denote the category with

\begin{enumerate}

\item objects the marked simplicial sets, and

\item morphisms the morphisms of simplicial sets which take marked edges to marked edges.

\end{enumerate}

Let $X,Y\in\Set_\Delta^+$ be marked simplicial sets. Then we denote by $\Map^\flat(X,Y)$ and $\Map^\sharp(X,Y)$ the simplicial sets with
$$
\Hom_{\Set_\Delta}(\Delta^n,\Map^\flat(X,Y))=\Hom_{\Set_\Delta^+}(X\times(\Delta^n)^\flat,Y),\text{ and }
$$
$$
\Hom_{\Set_\Delta}(\Delta^n,\Map^\sharp(X,Y))=\Hom_{\Set_\Delta^+}(X\times(\Delta^n)^\sharp,Y).
$$
Let $S$ be a simplicial set. Then a marked $S$-simplicial set is a marked simplicial set with a morphism to $S^\sharp$. We denote their category by $(\Set_\Delta^+)_{/S}$. Let $X,Y\in(\Set_\Delta^+)_{/S}$ be marked $S$-simplicial sets. Then $\Map^\flat_S(X,Y)\subseteq\Map^\flat(X,Y)$ is the simplicial subset on simplices $X\times(\Delta^n)^\flat\to Y$ such that the postcomposite with the structure map $Y\to S$ is the composite of the projection map and the structure map $X\times(\Delta^n)^\flat\to X\to S^\sharp$. We define $\Map^\sharp_S(X,Y)$ similarly.

Let $Z\xrightarrow pS$ be a Cartesian fibration. Then the marked $S$-simplicial set $Z^\natural$ has as marked edges the $p$-Cartesian edges. Let $X$ be a marked $S$-simplicial set. Then $\Map^\flat_S(X,Z^\natural)$ is a quasi-category, and $\Map^\sharp_S(X,Z^\natural)$ is its interior. Let $X\xrightarrow fY$ be a morphism of marked $S$-simplicial sets. Then the following are equivalent \cite{lurie2009higher}*{Proposition 3.1.3.3}.

\begin{enumerate}

\item For every Cartesian fibration $Z\to S$, the precomposition map
$$
\Map^\flat_S(Y,Z^\natural)\xrightarrow{\circ f}\Map^\flat_S(X,Z^\natural)
$$
is an equivalence of quasi-categories.

\item For every Cartesian fibration $Z\to S$, the precomposition map
$$
\Map^\sharp_S(Y,Z^\natural)\xrightarrow{\circ f}\Map^\sharp_S(X,Z^\natural)
$$
is a homotopy equivalence of Kan complexes.

\end{enumerate}
If these equivalent conditions are satisfied, then we say that \emph{$f$ is a Cartesian equivalence}. Now using the $\Map^\sharp_S$ as mapping spaces, we can equip $(\Set_\Delta^+)_{/S}$ with a left proper, combinatorial, simplicial model structure as follows \cite{lurie2009higher}*{Proposition 3.1.3.7, Corollary 3.1.4.4}.

\begin{enumerate}

\item The cofibrations are the maps which are monomorphisms on the underlying simplicial sets.

\item The weak equivalences are the Cartesian equivalences.

\end{enumerate}

This is called the \emph{Cartesian model structure}. Let $X\to S$ be a marked $S$-simplicial set. Then it is a fibrant object if and only if $X\simeq Y^\natural$ for some Cartesian fibration $Y\to S$ \cite{lurie2009higher}*{Proposition 3.1.4.1}. Let $X\xrightarrow fY$ be a morphism of Cartesian fibrations over $S$. Then $f$ is a Cartesian equivalence if and only if the fibres $X_s\xrightarrow{f_s}Y_S$ are categorical equivalences for all $s\in S$ \cite{lurie2009higher}*{Proposition 3.1.3.5}.

Let $\mathscr C$ be a quasi-category. Let $\mathscr E$ be a collection of edges containing the degenerate edges. Then $(\mathscr C,\mathscr E)$ is a marked simplicial set. Let $(\mathscr C,\mathscr E)\to\mathscr D^\natural$ be a fibrant replacement. Then we call $\mathscr D$ the \emph{simplicial localization of $\mathscr C$ with respect to $S$}, and write $\mathscr D=\mathscr C[\mathscr E^{-1}]$.

Note that the full simplicial subcategory of fibrant objects of $\Set_\Delta^+$ is precisely $\Cat_\infty^\Delta$. Equipping $(\Set_\Delta^+)^{\mathfrak C[S]}$ with the projective model structure induced by the Cartesian model structure, we get a Quillen equivalence \cite{lurie2009higher}*{Theorem 3.2.0.1}
$$
\adjoints{(\Set_\Delta^+)_{/S}}{(\Set_\Delta^+)^{\mathfrak C[S^\op]}}{\St^+}{\Un^+}.
$$
We refer to the functors $\St^+$ and $\Un^+$ as the \emph{straightening} and \emph{unstraightening} functors. Let $X\xrightarrow pS$ be a Cartesian fibration, and $S^\op\xrightarrow f\Cat_\infty$ a presheaf. Then by the $(\mathfrak C,N_\Delta)$-adjunction, $f$ corresponds to a simplicial functor $\mathfrak C[S^\op]\xrightarrow{f^\sharp}\Cat_\infty^\Delta\subseteq\Set_\Delta^+$. We say that \emph{$p$ is classified by $f$}, if $p\simeq\Un^+(f^\sharp)$.


\subsection{Homotopy limits and descent}\label{ss:descent}

Let $I$ and $\mathbf C$ be 1-categories. Consider the category of diagrams $\Fun(I,\bC)$. There is a constant functor
\[
{\rm const} : \bC\rightarrow \Fun(I,\bC).
\]
The right adjoint of the constant functor is the limit functor, when it exists. If
$\bC$ is a model category, we may take its right derived functor. In order to do
this, we need to equip $\Fun(I,\bC)$ with a model structure. We would like to
do this so that the adjunction ${\rm const} \dashv \lim$ becomes a Quillen pair.
This requires that the constant functor preserves cofibrations and acyclic cofibrations.
With this in mind we can try to put a model structure on $\Fun(I,\bC)$ where the
cofibrations and weak equivalences are defined objectwise on $I$. The fibrations
are then determined. When it exists this is the injective model structure on
the functor category and we may use it define a right derived functor of the limit
which is known as the \emph{homotopy limit}, written $\holim$.

We will be interested in the case where $\bC$ is $\Set_\Delta$ with its Quillen model structure.

\begin{definition}\label{defn:strict descent}
	Let $\bC$ be a Grothendieck site with products.
	Given a cover $U\rightarrow T$ we may form the Cech nerve $U^\bullet\rightarrow T$.
	A simplicial presheaf $\sF$ is a \emph{sheaf} if for
	every cover $U\rightarrow T$ we have
	\[
	\sF(T) \cong \holim \sF(U^\bullet).
	\]

\end{definition}

\begin{remark}
	This reduces to the ordinary definition in the case of a presheaf of sets. To reconcile this
    with the notion of stack, see \cite[6.5]{hollander}.
\end{remark}

\begin{warn}

Note that this notion does not include stacks, since it only includes functors, not pseudofunctors.

\end{warn}

We want to study descent conditions for presheaves of $\infty$-categories, which in our case will be maps $N(\mathbf C)^\op\xrightarrow{\mathcal F}\Cat_\infty$. As above, they can be formulated via homotopy limits. In the language of quasi-categories, it is easy to define these. Let $X$ be a quasi-category. Then $x\in X$ is a \emph{final object}, if for all $y\in X$, the mapping space $\Map_X(y,x)$ is contractible. This happens if and only if $x\in X$ is \emph{strongly final}, that is the restriction map $X_{/x}\to X$ is a trivial fibration \cite{lurie2009higher}*{Corollary 1.2.12.5}. Now let $K$ be a simplicial set, and $K\xrightarrow kX$ a diagram. Then the \emph{limit $\lim k$} is simply a final object in the overcategory $X_{/k}$. Note that these are automatically homotopy limits. Dually, we can define initial objects and colimits.

Following the straightening-unstraightening construction recalled in \S\ref{ss:Cartesian}, a presheaf $N(\mathbf C)^\op\xrightarrow{\mathcal F}\Cat_\infty$ classifies a Cartesian fibration $\mathscr X\xrightarrow pN(\mathbf C)$. More generally, we can replace $N(\mathbf C)$ with some simplicial set $K$. In this case, there is a generalization of the description of the homotopy limit of a fibrant cosimplicial space \cite{bousfield1972homotopy}*{X,\S3} in terms of the quasi-category of Cartesian sections which we will define now.

\begin{notn}

Let $\mathscr X\xrightarrow pK$ be an inner fibration.  The quasi-category of sections, denoted $\Gamma(K,\mathscr X)$,
is  the pullback
\begin{center}

\begin{tikzpicture}[xscale=3,yscale=1.5]
\node (C') at (0,1) {$\Gamma(K,\mathscr X)$};
\node (D') at (1,1) {$\Fun(K,\mathscr X)$};
\node (C) at (0,0) {$\{\id_K\}$};
\node (D) at (1,0) {$\Fun(K,K)$.};
\path[->,font=\scriptsize,>=angle 90]
(C') edge (D')
(C') edge (C)
(D') edge node [right] {$p\circ$} (D)
(C) edge (D);
\end{tikzpicture}

\end{center}
We will say that a section $\sigma$, that is a $0$-simplex $\sigma\in \Gamma(K,\mathscr X)$,
is \emph{Cartesian} if $\sigma(e)$ is Cartesian for each edge $e\in K_1$.

Let us denote by $\Gamma_{\Cart}(K,\mathscr X)\subset\Gamma(K,\mathscr X)$ the full subcategory on Cartesian sections. Then the limit result is the following. Let $\mathscr X\xrightarrow pK$ be a Cartesian fibration classified by a map $K^\op\xrightarrow k\Cat_\infty$. Then we have \cite{lurie2009higher}*{Corollary 3.3.3.2}
$$
\lim k=\Gamma_{\Cart}(K,\mathscr X).
$$

Let $L\xrightarrow kK$ be a map of simplicial sets. Then we write the pullback $L\times_K\mathscr X$ as $p|k$, and we write $\Gamma(k,\mathscr X)=\Gamma(L,p|k)$ and $\Gamma_{\Cart}(k,\mathscr X)=\Gamma_{\Cart}(L,p|k)$.

Suppose that $S\in L$ is a final object. We will usually notation such as $L\xrightarrow{X_\bullet}\mathscr X$ to denote sections so that we can let $X_S=X$.

\end{notn}

Recall \cite[4.1.1.1]{lurie2009higher} that a morphism $L\xrightarrow kK$ of
simplicial sets is \emph{cofinal} if for any right fibration $X\rightarrow K$ we have
a homotopy equivalence
\[
\Gamma(K,X) \xrightarrow{\circ k} \Gamma(L,X|k).
\]

An example of a cofinal morphism is the inclusion $\{S\}\hookrightarrow \mathscr C$
of a final object of a quasi-category. This follows from the quasi-categorical version of
Quillen's theorem A, see loc. cit 4.1.3.1.

\begin{lemma}\label{l:cart}\label{lem:Cartesian section}
Let $X\rightarrow K$ be a Cartesian fibration over a quasi-category. Let $S\in K$ be a final
object. Then the natural restriction
$$\Gamma_{\Cart}(K,X)\to X_S$$
is a homotopy equivalence.
\end{lemma}

\begin{proof}
Let $X_{\Cart} \subseteq X$ be the subcategory generated by Cartesian edges.
In other words $X_{\Cart}$ and $X$ have the same $0$-simplicies. For $n>0$, the
$n$-simplices of $X_{\Cart}$ are those $n$-simplicies of $X$ whose edges are
Cartesian.
Then
$$
X_{\Cart}\rightarrow K $$
is a right fibration by \cite[2.4.2.5]{lurie2009higher}.
The result follows from definitions now.
\end{proof}

\begin{defn}\label{d:descent}

Let $\mathscr X\xrightarrow pK$ be a Cartesian fibration, and $\Delta_+^\op\xrightarrow{k_+}K$ an augmented simplicial object. Let us denote the restriction $k_+|\Delta^\op$ by $k$. The restriction map $\Gamma_{\Cart}(k_+,\mathscr X)\to\Gamma(k_+(-1),\mathscr X)$ is a trivial fibration by Lemma \ref{l:cart}. Therefore the zigzag
$$
\Gamma(k_+(-1),\mathscr X)\xleftarrow\simeq\Gamma_{\Cart}(k_+,\mathscr X)\to\Gamma_{\Cart}(k,\mathscr X)
$$
gives a map $\Gamma(k_+(-1),\mathscr X)\xrightarrow{La^*}\Gamma_{\Cart}(k,\mathscr X)$. We say that \emph{$p$ satisfies descent along $k_+$}, if the functor $La^*$ is an equivalence of quasi-categories. Let $U\xrightarrow gT$ be an edge in $K$. Then we say that \emph{$p$ satisfies descent along $g$}, if $p$ satisfies descent along the augmented \v Cech nerve $\check C(g)_+$.

We will reconcile this definition with Definition \ref{defn:strict descent} in Remark \ref{r:descent} below. Note also, that as per our conventions,
we have omitted the nerve in our notation in various places, for example
$N(\Delta^\op)$.
\end{defn}

\begin{rem}\label{r:descent}
Let $\Delta^\op_+\xrightarrow{k_+} K$ be an augmented simplicial object. Let $\Delta\xrightarrow{\st(p|k)}\Cat_\infty$ be the straightening of $p|k$. Then the straightening
 $\Delta_+\xrightarrow{\st(p|k_+)}$ of $p|k_+$ is a cone over $\st(p|k)$, that is a point of $(\Cat_\infty)_{/\st(p|k)}$. We obtain a functor of quasi-categories

 $$\Gamma(k_+(-1),\mathscr X)\xrightarrow{(La^*)'}\holim \st(p|k).$$
Descent is usually phrased by asserting that $(La^*)'$ is a weak equivalence. It is equivalent to Definition \ref{d:descent} as $\holim \st(p|k)\simeq\Gamma_{\Cart}(k,p)$ \cite{lurie2009higher}*{Corollary 3.3.3.2}.
\end{rem}


\section{Background on quasi-categories of complexes}


\subsection{The dg-nerve and the bounded below derived category}
\label{ss:dgNerve}

This construction is from \cite[Ch. 1]{lurie2014higher}. We will work exclusively
with complexes whose differential has degree $+1$, so it is worth recalling the
definition in our context here. One can pass from a cohomological complex $C^\bullet$
to a homological complex $C_\bullet$ by setting $C_n = C^{-n}$.

Given a $\dg$-category $\bC$ we will denote the (cohomological) mapping complex between
a pair of objects by $\Hom^\bullet(x,y)$.
We may apply our reindexing construction to $\bC$  to
obtain a dg-category with homological mapping complexes. Lurie's dg-nerve construction
may then be applied to this category to obtain a quasi-category.
Lets unwind definitions to see what we obtain. The $n$-simplicies are pairs
$((X_i)_{0\le i \le n}, \{ f_I\}_{I\in \str([n])})$ where
\begin{itemize}
\item $X_i$ are objects of $\bC$
\item $\str([n])$ is the collection of subsets of $[n]$ length at least $2$
\item $f_I\in \Hom^{2-|I|}(X_{\min(I)}, X_{\max(I)})$.
\end{itemize}
 This data is subject to the condition that  for each $I\in\str[n]$
 of the form $I=\{ i_{\min}<i_1<\ldots i_m<i_{\max}  \}$
 we have
 \[
 df_I =\sum_{1\le j \le m}(-1)^j (f_{I-\{i_j\}}  - f_{\{ i_j<\ldots <i_m<i_{\max} \}} \circ f_{\{ i_{\min}<i_1\ldots <i_j\} }).
 \]
These collections acquire the structure of a simplicial set by defining
$$ \alpha^*(((X_i)_{0\le i \le n}, \{ f_I\}_{I\in \str([n])}))
= (X_{\alpha(j)})_{0\le j \le m}, (g_I)_{I\in \str[m]})
$$ for
an order preserving function $\alpha : [m]\rightarrow [n]$ where
\[
g_I =
\begin{cases}
f_{\alpha(I)} & \mbox{if }\alpha|_I\mbox{ is injective} \\
1_{X_i} & |I|=2\mbox{ and } \alpha(I)=\{i\} \\
0 & \mbox{otherwise}.
\end{cases}
\]

Consider now a Grothendieck abelian category $\bA$. We may consider the $\dg$-category
 $\Ch^+(\bA)$  whose objects are bounded below chain complexes.

 For a pair of chain complexes $A$ and $B$ let's write $\Hom_\dg(A,B)$ for the chain
complex of maps between them.
In degree $p$ it is
\[
\Hom_\dg(A,B)^p = \prod_n \Hom(A^n, B^{n+p}).
\]
Given $f\in \Hom(A^n, B^{n+p})$,
the differential is given by the formula $d(f) =  d_B \circ f - (-1)^p f\circ d_A$.
The choice of sign insures that the $0$-cycles are chain maps.

Applying the
dg-nerve we obtain a quasi-category $\ndg(Ch^+(\bA))$.
We will be interested in the full subcategory on complexes of injectives,
$\mathscr D^+(\bA)=\ndg(\Ch^+(\bA-inj))$
Let's recall what this
simplicial set looks like.

The $n$-simplicies consist of pairs $(\{K_i\}, \{f_I\})$ where
\begin{enumerate}
	\item for $0\le i \le n$ we have a bounded below chain complex $I_i$ of injectives
	\item for each subset $I\subseteq \{0,\ldots ,n\}$ of the form
	$I=\{ i_- < i_1<\cdots < i_m<i_+\}$ where $m>0$ we have $f_I\in \Hom_\dg(I_{i_-},I_{i_+})^{2-|I|}$
satisfying the equation
\[
df_I = \sum_{1\le j\le m} (-1)^j (f_{I-\{i_j\}}-
f_{ i_j< \cdots < i_m < i_+} \circ f_{i_-<i_1<\cdots < i_j}).
\]
\end{enumerate}

Now the homotopy category of this simplicial set, \cite[page 29]{lurie2009higher},
is exactly the homotopy category of bounded below complexes of injectives. This
is the ordinary derived category. For this reason we call $\msD^+(\bA)$
the bounded below derived quasi-category.

There is a useful description of mapping spaces in a dg-nerve using Dold--Kan complexes, which we recall now. Let $A$ be a simplicial abelian group. Then its \emph{Moore complex} is the strictly connective chain complex with $(CA)_n=A_n$ and $d^{CA}_n=\sum_{k=0}^n(-1)^kd^A_k$ for $n\ge0$. The \emph{normalized chain complex} is the subcomplex with $(NA)_n=\cap_{k=1}^n\ker d^A_kA_n\subseteq CA_n$ for $n\ge0$. It turns out that the functor $\Ab^{\Delta^\op}\xrightarrow{A\mapsto NA}\Ch^{\ge0}(\mathbf Z)$ is an equivalence of categories \cite{goerss1999simplicial}*{Corollary III.2.3}. It has a canonical quasi-inverse, the \emph{Dold--Kan complex} functor. One of its descriptions\footnote{By abuse of notation, for a simplicial set $K$, we will also denote the free simplicial abelian group $\mathbf Z[K]$ by $K$.} is to make $(\DK A)_n=\Hom_{\Ch(\mathbf Z)}(N\Delta^n,A)$, and let the structure come from the cosimplicial structure of $N\Delta^\bullet$. It turns out that there exists an isomorphism of simplicial abelian groups $(DK A)_n\xrightarrow{\Psi_n}\bigoplus_{[n]\twoheadrightarrow[k]}A_k$ \cite{goerss1999simplicial}*{Proposition III.2.2}, but its inverse needs to be constructed via a recursive process, not explicitly. We propose the following alternative description, based on the fact that, letting $(DA)_n\subset(CA)_n$ denote the subcomplex on degenerate simplices, the inclusion map induces an isomorphism $NA\to CA/DA$ \cite{goerss1999simplicial}*{Theorem III.2.1}.

\begin{defn}

For $n\ge0$, the \emph{Koszul complex} $K\Delta^n$ is the subcomplex $K\Delta^n\subseteq C\Delta^n$ on nondegenerate simplices.

\end{defn}

\begin{rem}

We call this the Koszul complex, since $K\Delta^n$ is the na\"ive truncation and shift to the positive part of the Koszul complex $K(\mathbf Z^n)$.

\end{rem}

\begin{prop} (1) The inclusion map induces an isomorphism $K\Delta^n\to C\Delta^n/D\Delta^n$.

(2) Let $A$ be a strictly connective chain complex. Then for any $n\ge0$, restriction to simplices of the form $\{0\}\cup\sigma$ for $\sigma\in\{1,\dotsc,n\}$ gives an isomorphism $\Hom_{\Ch(\mathbf Z)}(K\Delta^n,A)\to\bigoplus_{\sigma\in\{1,\dotsc,n\}}A_{|\sigma|}$.

\end{prop}

\begin{lemma}
\begin{enumerate}
\item Let $\bC$ be a $\dg$-category and $X$ and $Y$ objects of $\bC$. Then $\pi_n(\Map_{N_\dg(\bC)}(X,Y)) = H^{-n}(\Hom^\bullet(X,Y)).$
\item
We denote by $\msD^{\ge 0}(\bA)$ the full subcategory on complexes that are acyclic
in negative degrees.
Let $K\in \msD^{\ge 0}(\bA)$ and $J\in \msD^+(\bA)$. The the canonical morphism
to the good truncation
\[
J\rightarrow \tau_{\ge 0} J
\]
induces a weak equivalence
\[
\Map(\tau_{\ge  0} J, K) \rightarrow \Map(J,K)
\]
\end{enumerate}
\end{lemma}

\begin{proof}
The proof of \cite[Proposition 1.3.1.17]{lurie2014higher} shows that
\[
\pi_n(\Map_{N_\dg(\bC)}(X,Y)) = \pi_n(\DK\tau_{\ge 0} \Hom^{-\bullet}(X,Y)).
\]
Here $\DK$ is the Dold--Kan functor. The result follows from the fact that the Dold--Kan
functor identifies cohomology with homotopy groups.

For the second part, as the mapping space is invariant under equivalence, we may assume
that $K$ is concentrated in non-negative degrees. The result follows easily from the previous
part.
\end{proof}

The most important feature of the quasi-category $\mathscr D^+(\bA)$  is that it is in fact a stable quasi-category,
\cite[1.3]{lurie2014higher}.


\subsection{Cotorsion pairs and unbounded complexes}\label{ss:cotorsion pairs}

We want to get a presentable fibration ${}^\op\mathscr D_S\to\Sch_S$, because they have a nice theory which we want to use. That is, we want the fibres $\mathscr D(T)^\op$ to be presentable quasi-categories. Presentable quasi-categories have all small limits and colimits. Therefore, we will need to consider unbounded complexes. For a morphism of $S$-schemes $U\xrightarrow gT$, we will also want to get an adjoint pair $\smalladjoints{\mathscr D(T)}{\mathscr D(U)}{Lg^*}{Rg_*}$. To define these functors, we will need to restrict to complexes on which the functors $g^*$ resp.~$g_*$ are \emph{homotopical}, that is they take equivalences (which in this setting are exactly the quasi-isomorphisms) to equivalences. That is, we will need to get functorial dg-flat resp.~dg-injective resolutions. To get these, we will employ two model structures given in \cite{gillespie2007kaplansky}.

Let $\mathscr G$ be an abelian category. Let $\mathscr A,\mathscr B\subseteq\mathscr G$. Let
$$
\mathscr A^\perp=\{X\in\mathscr G:\Ext^1(A,X)=0\text{ for all }A\in\mathscr A\},\text{ and }
{}^\perp\mathscr B=\{X\in\mathscr G:\Ext^1(X,B)=0\text{ for all }B\in\mathscr B\}.
$$
Then $(\mathscr A,\mathscr B)$ is a \emph{cotorsion pair}, if $\mathscr A^\perp=\mathscr B$, and $\mathscr A={}^\perp\mathscr B$.

Let $T$ be an $S$-scheme and let $\mathscr G=\mathscr O_T\text{-Mod}$. Then we have two important cotorsion pairs:
\begin{itemize}

\item $(\mathscr O_T\text{-Mod},\mathscr I)$, where $\mathscr I$ is the class of injective $\mathscr O_T$-modules, and

\item $(\mathscr F,\mathscr C)$, where $\mathscr F$ is the class of flat $\mathscr O_T$-modules, and $\mathscr C$ is the class of cotorsion $\mathscr O_T$-modules.

\end{itemize}

Let $X$ be a cochain complex in $\mathscr G$. Then

\begin{enumerate}

\item $X$ is an \emph{$\mathscr A$-complex}, if it is exact, and $Z^nX\in\mathscr A$ for all $n$. The collection of $\mathscr A$-complexes is denoted by $\Tilde{\mathscr A}$.

\item $X$ is a \emph{$\mathscr B$-complex}, if it is exact, and $Z^nX\in\mathscr B$ for all $n$. The collection of $\mathscr B$-complexes is denoted by $\Tilde{\mathscr B}$.

\item $X$ is a \emph{dg-$\mathscr A$-complex}, if $X^n\in\mathscr A$ for all $n$, and for every map $X\xrightarrow fB$, if $B$ is a $\mathscr B$-complex, then $f$ is nullhomotopic. The collection of dg-$\mathscr A$-complexes is denoted by $\dg\Tilde{\mathscr A}$.

\item $X$ is a \emph{dg-$\mathscr B$-complex}, if $X^n\in\mathscr B$ for all $n$, and for every map $A\xrightarrow fX$, if $A$ is a $\mathscr A$-complex, then $f$ is nullhomotopic. The collection of dg-$\mathscr B$-complexes is denoted by $\dg\Tilde{\mathscr B}$.

\end{enumerate}

A complex $I$ of $\mathscr O_T$-modules is called \emph{dg-injective}, if it is a dg-$\mathscr I$-complex. A complex $P$ of $\mathscr O_T$-modules is called \emph{dg-flat}, if it is a dg-$\mathscr F$-complex.

If certain conditions are satisfied \cite{gillespie2007kaplansky}*{Theorem 4.12}, then we get a model structure on $\Ch\mathscr G$ such that

\begin{itemize}

\item the weak equivalences are the quasi-isomorphisms,

\item the cofibrations (resp.~trivial cofibrations) are the monomorphisms whose cokernels are in $\dg\Tilde{\mathscr A}$ (resp.~$\Tilde{\mathscr A}$), and

\item the fibrations (resp.~trivial fibrations) are the epimorphisms whose kernels are in $\dg\Tilde{\mathscr B}$ (resp.~$\Tilde{\mathscr B}$).

\end{itemize}

In the case $\mathscr G=\mathscr O_T\text{-Mod}$, the cotorsion pairs $(\mathscr O_T\text{-Mod},\mathscr I)$ and $(\mathscr F,\mathscr C)$ satisfy these conditions \cite{gillespie2007kaplansky}*{Corollaries 7.1 and 7.8}. The model structures on $\Ch(T)$ we get we call the \emph{injective model structure}, and the \emph{flat model structure}, respectively.

\begin{notn}

Let $T$ be an $S$-scheme. We denote by $\Ch_{\mathrm{qc}}(T)\subseteq\Ch(T)$ the full sub-dg-category on complexes of $\mathscr O_T$-modules with quasi-coherent cohomology. Let $\Ch_{\mathrm{qc},\mathrm{inj}}(T),\Ch_{\mathrm{qc},\mathrm{fl}}(T)$ be the full sub-dg-categories of dg-injective and dg-flat complexes respectively. Let's denote their dg-nerves by $\mathscr D(T)$ and $\mathscr D_{\mathrm{fl}}(T)$, respectively. Let's denote $N_\dg\Ch_{\mathrm{qc}}(T)$ by $\mathscr Ch(T)$.

\end{notn}

A pair of functors $U:\mathscr C \rightarrow \mathscr D$ and $F:\mathscr D\rightarrow \mathscr C$ between quasi-categories are \emph{adjoint}, if
there is a unit transformation $u: id_{\mathscr D}\rightarrow U\circ F$ such that
the composition
\[
\Map_{\mathscr C}(Fx,y)\rightarrow \Map_{\mathscr D}(U(Fx),U(y))
\xrightarrow{u}  \Map_{\mathscr D}(x,U(y))
\]
is a weak equivalence, see \cite[5.2.2.7,5.2.2.8]{lurie2009higher}.

The dg-injective resolution functor $\mathscr Ch(T)\xrightarrow{I_T}\mathscr D(T)$ will be a left adjoint to the inclusion $\mathscr D(T)\to\mathscr Ch(T)$. That is, $I_T$ will be a \emph{localization} functor. To show that such a functor exists, it is enough to show \cite{lurie2009higher}*{Proposition 5.2.7.8} that every complex $E\in\mathscr Ch(T)$ admits a \emph{$\mathscr D(T)$-localization}, that is there exists a dg-injective complex $I\in\mathscr D(T)$ and a morphism $E\xrightarrow qI$ such that for all dg-injective complexes $J\in\mathscr D(T)$, the precomposition map
$$
\Map_T(I,J)\xrightarrow{\circ q}\Map_T(E,J)
$$
is a weak equivalence.

Similarly, the dg-flat resolution functor $\mathscr Ch(T)\xrightarrow{P_T}\mathscr D_{\mathrm{fl}}(T)$ will be a right adjoint to the inclusion $\mathscr D_{\mathrm{fl}}(T)\to\mathscr Ch(T)$. To get it, we will need to show that every complex $E\in\mathscr Ch(T)$ admits a \emph{$\mathscr D_{\mathrm{fl}}(T)$-colocalization}, that is there exists a dg-flat complex $P\in\mathscr D_{\mathrm{fl}}(T)$ and a morphism $P\xrightarrow rE$ such that for all dg-flat complexes $Q\in\mathscr D_{\mathrm{fl}}(T)$, the postcomposition map
$$
\Map_T(Q,P)\xrightarrow{q\circ}\Map_T(Q,E)
$$
is a weak equivalence.

\begin{prop}\label{prop:inj, flat res}

Let $T$ be an $S$-scheme. Let $E\in\mathscr Ch(T)$. Then $E$ admits a $\mathscr D(T)$-localization, and a $\mathscr D_{\mathrm{fl}}(T)$-colocalization.

\end{prop}

\begin{proof}

Applying Lemma \ref{lem:cotorsion to resolution} in the situation of \cite{gillespie2007kaplansky}*{Corollary 7.8} gives a $\mathscr D_{\mathrm{fl}}(T)$-colocalization. Applying the dual argument of Lemma \ref{lem:cotorsion to resolution} in the situation of \cite{gillespie2007kaplansky}*{Corollary 7.1} gives a $\mathscr D(T)$-localization.

\end{proof}

\begin{lem}\label{lem:cotorsion to resolution}

Let $\mathscr A$ be a class of objects in a Grothendieck category $\mathscr G$ satisfying the conditions of \cite{gillespie2007kaplansky}*{Theorem 4.12}. Let $\mathscr B=\mathscr A^\perp$. Let $E\in\Ch(\mathscr G)$. Then there exists a fibrant resolution $P\xrightarrow rE$. For any dg-$\Tilde{\mathscr A}$ complex $Q$, the postcomposition map between the derived Hom complexes
$$
\RHom(Q,P)\to\RHom(Q,E)
$$
is a quasi-isomorphism.

\end{lem}

\begin{proof}

Let $K$ be the kernel of $r$. Then we get a distinguished triangle $K\to P\to E\to$ in $\mathbf D(T)$ \cite{weibel1994introduction}*{Example 10.4.9}. Therefore, we get a long exact sequence
$$
\dotsb\to\Ext^n(Q,K)\to\Ext^n(Q,P)\to\Ext^n(Q,E)\to\Ext^{n+1}(Q,K)\to\dotsb
$$
By construction, $K$ is a $\Tilde{\mathscr B}$-complex. Therefore, every map $Q\to K$ is nullhomotopic. This implies that $\Ext^m(Q,K)=0$ for all $m$. This proves the claim.

\end{proof}

\begin{notn}\label{notn:derived pullback and pushforward}

Let $T$ be an $S$-scheme. By Proposition \ref{prop:inj, flat res}, we get an injective resolution functor $\mathscr Ch(T)\xrightarrow{I_T}\mathscr D(T)$ with unit map $\mathscr Ch(T)\times\Delta^1\xrightarrow{q_T}\mathscr Ch(T)$, and a flat resolution functor $\mathscr Ch(T)\xrightarrow{P_T}\mathscr D_{\mathrm{fl}}(T)$ with counit map $\mathscr Ch(T)\times\Delta^1\xrightarrow{r_T}\mathscr Ch(T)$. We will often drop the subscript $T$ from the notation.

Let $U\xrightarrow gT$ be a morphism of $S$-schemes. Then we get a derived pullback functor $\mathscr D(T)\xrightarrow{Lg^*=g^*\circ P_T}\mathscr D(U)$ and a derived pushforward functor $\mathscr D(U)\xrightarrow{Rg_*=g_*\circ I_U}\mathscr D(T)$.

\end{notn}

\begin{proposition}\label{p:adjointFunc}
Let $U\xrightarrow gT$ be a morphism of schemes.
\begin{enumerate}
\item The derived pullback functor $\mathscr D(T)\xrightarrow{Lg^*}\mathscr D(U)$ is left adjoint to the derived pushforward functor $\mathscr D(U)\xrightarrow{Rg_*}\mathscr D(T)$.
\item The functor $Rg_*$ restricts to a funtor
$\msD^{\ge n}(U)\rightarrow \msD^{\ge n}(T)$.
\end{enumerate}
\end{proposition}

\begin{proof} (1) We have a natural quasi-isomorphism $\RHom_U(Lg^*I,J)\xrightarrow\alpha\RHom_V(I,Rg_*J)$ \cite{lipman2009notes}*{Proposition 3.2.3}, which gives a natural equivalence $\Map_U(Lg^*I,J)\xrightarrow{\DK\tau_{\le0}\alpha}\Map_V(I,Rg_*J)$ \cite{schwede2003equivalences}*{\S4.1}. This shows that $(Lf^*,Rg_*)$ gives an adjoint pair of simplicial categories $\smalladjoints{\mathfrak C\mathscr D(U)}{\mathfrak C\mathscr D(T)}{\mathfrak CLf^*}{\mathfrak CRg_*}$ \cite{lurie2014higher}*{Proposition 1.3.1.17}. Therefore, a functor equivalent to $Lf^*$ is left adjoint to a functor equivalent to $Rg_*$ \cite{lurie2009higher}*{Corollary 5.2.4.5}, and that is enough \cite{lurie2009higher}*{Proposition 5.2.1.4}.

(2) Let $I\in\mathscr D^{\ge n}(U)$. Then $I\simeq\tau_{\ge n}I$. By the Cartan--Eilenberg resolution, we get an injective complex $J$ such that $J\simeq\tau_{\ge n}I$, and $J^m=0$ for $m<n$. Since we have $I(I)\simeq J$, a zigzag of quasi-isomorphisms between dg-injective complexes, we get $Rg_*I=g_*I(I)\simeq g_*J$. Here, $g_*J\in\mathscr D^{\ge n}(T)$ by construction. Therefore, we get $Rg_*I\in\mathscr D^{\ge n}(T)$.

\end{proof}


\section{Twisted arrow categories and mapping spaces}\label{s:twisted}

Let $\bC$ be an ordinary category. The twisted arrow category of $\bC$ is the category denoted $\Tw(\bC)$,
whose objects are arrows in $\bC$. A morphism in $\Tw(\bC)$ from $m'\xrightarrow{\alpha'}n'$ to $m\xrightarrow\alpha n$ amounts
to a commutative diagram in $\bC$ of the form
\begin{center}

\begin{tikzpicture}[xscale=2,yscale=1.5]
\node (C') at (0,1) {$m'$};
\node (D') at (1,1) {$m$};
\node (C) at (0,0) {$n'$};
\node (D) at (1,0) {$n$.};
\path[->,font=\scriptsize,>=angle 90]
(C') edge node [above] {$\mu$} (D')
(C') edge node [right] {$\alpha'$} (C)
(D') edge node [right] {$\alpha$} (D)
(D) edge node [above] {$\nu$} (C);
\end{tikzpicture}

\end{center}
Notice that morphisms $\mu$ and $\nu$ are in opposite directions, so that
we have a functor $\Tw(\bC)\rightarrow \bC\times \bC^\op$.

 This construction
has been  made for quasi-categories, see \cite[5.2]{lurie2014higher}. Let us briefly recall it.
If $I$ and $J$ are finite ordered sets we can form the ordered set $I\star J$. The underlying
set of $I\star J$ is the disjoint union of $I$ and $J$ and the elements of $I$ are placed before
those of $J$ in the ordering. If $S$ is a simplicial set, the $n$ simplices of
$\Tw(S)$ are $S([n]\star[n]^\op)$. The op is important for the simplicial structure.

The main application of the twisted arrow category is that it is a right fibration classified by the mapping space functor. Let $\mathscr C$ be a quasi-category. Then the canonical map $\Tw\mathscr C\xrightarrow\lambda\mathscr C\times\mathscr C^\op$ is a right fibration
\cite{lurie2014higher}*{Proposition 5.2.1.3}. That is, it is classified by a presheaf $\mathscr C^\op\times\mathscr C\to\mathscr S$. Then the corresponding map $\mathscr C\to\mathscr P(\mathscr C)$ is equivalent to the Yoneda embedding \cite{lurie2014higher}*{5.2.1.11}. Therefore, we will denote the map $\mathscr C^\op\times\mathscr C\to\mathscr S$ by $\Map_{\mathscr C}$.

This can be used to give formulas for mapping spaces in functor categories. Let $F,G\colon K\rightrightarrows\mathscr X$ be functors of quasi-categories. Then we have
$$
\Map_{\Fun(K,\mathscr C)}(F,G)=\lim((\Tw K)^\op\xrightarrow{\lambda^\op}K^\op\times K\xrightarrow{F^\op\times G}\mathscr X^\op\times\mathscr X\xrightarrow{\Map_{\mathscr X}}\mathscr S)
$$
\cite{gepner2017lax}*{Definition 2.5, Proposition 5.1}. Here, by the limit of the functor we mean the limit of the diagram of spaces $(\Tw\mathscr X)^\op\xrightarrow{(m\xrightarrow{\alpha}n)\mapsto\Map_{\mathscr X}(F_m,G_m)}\mathscr S$.

Based on this result, we will give formulas for mapping spaces in section categories, Proposition \ref{prop:Map of Gamma}:
$$
\Map_{\Gamma(K,\mathscr X)}(F,G)=\lim((\Tw K)^\op\xrightarrow{(m\xrightarrow\alpha n)\mapsto\Map_{\alpha}(F_m,G_n)}\mathscr S),
$$
and Cartesian section categories, Proposition \ref{prop:Map of Gamma_Cart}:
$$
\Map_{\Gamma_{\Cart}(K,\mathscr X)}(F,G)=\lim(K^\op\xrightarrow{m\mapsto\Map_m(F_m,G_m)}\mathscr S).
$$

\subsection{Mapping spaces in the section quasi-category}

\begin{rem}

Let $L\xrightarrow kK$ be a morphism of simplicial sets, and $\mathscr X\to K$ an inner fibration. In the notation of \cite{lurie2009higher}, we have $\Gamma(k,\mathscr X)=\Map_K(L,\mathscr X)$.

\end{rem}

\begin{lem}\label{lem:twisted arrow for Fun}

Let $\mathscr X\xrightarrow pK$ be an inner fibration and $F,G\in\Gamma(K,\mathscr X)$ two sections. Then we can form the fibre product of the maps
$$
\Tw K\xrightarrow{\lambda_K}K\times K^\op\xrightarrow{F\times G^\op}\mathscr X\times\mathscr X^\op\text{ and }\Tw\mathscr X\xrightarrow{\lambda_{\mathscr X}}\mathscr X\times\mathscr X^\op.
$$
Then the induced map
$$
\Tw K\times_{\mathscr X\times\mathscr X^\op}\Tw\mathscr X\xrightarrow{\Tw K\times_{\mathscr X\times\mathscr X^\op}\Tw p}\Tw K\times_{\mathscr X\times\mathscr X^\op}\Tw K=\Tw K\times_{K\times K^\op}\Tw K
$$
is a right fibration.

\end{lem}

\begin{proof}

Let $n\ge1$ and $0<k\le n$. Let $\Delta^{2n+1}\xrightarrow{\mathrm{mir}}(\Delta^{2n+1})^\op$ denote the mirror map $j\mapsto 2n+1-j$. We need to solve the lifting problem
\begin{center}

\begin{tikzpicture}[xscale=4,yscale=1.5]
\node (C') at (0,1) {$\Lambda^n_k$};
\node (D') at (1,1) {$\Tw K\times_{\mathscr X\times\mathscr X^\op}\Tw\mathscr X$};
\node (C) at (0,0) {$\Delta^n$};
\node (D) at (1,0) {$\Tw K\times_{K\times K^\op}\Tw K$,};
\path[->,font=\scriptsize,>=angle 90]
(C') edge node [above] {$(\tau|\Lambda^n_k,\sigma)$} (D')
(C') edge (C)
(D') edge (D)
(C) edge node [below] {$(\tau,\Bar\sigma)$} (D)
(C) edge [dashed] (D');
\end{tikzpicture}

\end{center}
that is we need to find $\Delta^n\xrightarrow{\Tilde\sigma}\Tw\mathscr X$ such that $\Tilde\sigma|\Lambda^n_k=\sigma,\,\Tw p\circ\Tilde\sigma=\Bar\sigma$ and $\lambda_{\mathscr X}\circ\Tilde\sigma=(F\times G^\op)\circ\lambda_K\circ\tau=(F\times G^\op)\circ\lambda_K\circ\Bar\sigma$. Let's rephrase this as a lifting problem
\begin{center}

\begin{tikzpicture}[xscale=3,yscale=1.5]
\node (C') at (0,1) {$P$};
\node (D') at (1,1) {$\mathscr X$};
\node (C) at (0,0) {$\Delta^{2n+1}$};
\node (D) at (1,0) {$K$,};
\path[->,font=\scriptsize,>=angle 90]
(C') edge (D')
(C') edge (C)
(D') edge (D)
(C) edge (D)
(C) edge [dashed] (D');
\end{tikzpicture}

\end{center}
that is let's find the simplicial subset $P\subseteq\Delta^{2n+1}$ on which the value in $\mathscr X$ has been fixed. The map $\Lambda^n_k\xrightarrow\sigma\Tw\mathscr X$ corresponds by construction to a map $(\Lambda^n_k)\star(\Lambda^n_k)^\op\to\mathscr X$. The map $\Delta^n\xrightarrow{(F\times G^\op)\circ\lambda_K\circ\Bar\sigma}\mathscr X\times\mathscr X^\op$ corresponds to a map $\Delta^{[0,n]}\cup\Delta^{[n+1,2n+1]}\to\mathscr X$. Therefore, we have $P=((\Lambda^n_k)\star(\Lambda^n_k)^\op)\cup\Delta^{[0,n]}\cup\Delta^{[n+1,2n+1]}$. Let $T$ be the vertex set of a face of $P$. Then it is one of the following two types.

\begin{enumerate}

\item We have $([0,n]\setminus\{k\})\not\subseteq T$ and $([n+1,2n+1]\setminus\{2n+1-k\})\not\subseteq T$.

\item We have $T\subseteq[0,n]$ or $T\subseteq[n+1,2n+1]$.

\end{enumerate}

That is, $P\subset\Delta^{2n+1}$ is the largest simplicial subset which does not have any of the following faces.

\begin{itemize}

\item Faces with vertex set $([0,n]-\{k\})\cup S'$ for a nonempty subset $S'\subseteq[n+1,2n+1]$.

\item Faces with vertex set $S''\cup([n+1,2n+1]-\{2n+1-k\})$ for a nonempty subset $S''\subseteq[0,n]$.

\end{itemize}

Consider the chain of inclusions
$$
P=P_0\subset P_1\subset\dotsb\subset P_{n+1}=\Delta^{2n+1},
$$
where $P_\ell\subseteq\Delta^{2n+1}$ is the largest sub-simplicial set which does not have any of the following faces.

\begin{itemize}

\item Faces with vertex set $([0,n]-\{k\})\cup S'$ for a subset $S'\subseteq[n+1,2n+1]$ of size larger than $\ell$.

\item Faces with vertex set $S''\cup([n+1,2n+1]-\{2n+1-k\})$ for a subset $S''\subseteq[0,n]$ of size larger than $\ell$.

\end{itemize}

We will solve the lifting problem by ascending this chain one by one. Suppose that we have already lifted to a map $P_{\ell}\to\mathscr X$ for some $0\le\ell\le n$. Let
$$
\{S'\subseteq[n+1,2n+1]:|S'|=\ell+1\}=\{S'_1,\dotsc,S'_{\binom{n+1}{\ell+1}}\}\text{ and }
\{S''\subseteq[0,n]:|S''|=\ell+1\}=\{S''_1,\dotsc,S''_{\binom{n+1}{\ell+1}}\}.
$$
Consider the chain of inclusions
$$
P_{\ell}=P'_{\ell,0}\subset P'_{\ell,1}\subset\dotsb\subset P'_{\ell,\binom{n+1}{\ell+1}}=P''_{\ell,0}\subset P''_{\ell,1}\subset\dotsb\subset P''_{\ell,\binom{n+1}{\ell+1}}=P_{\ell+1},
$$
where $P'_{\ell,i}\subseteq\Delta^{2n+1}$ is the largest sub-simplicial set which does not have any of the following faces.

\begin{itemize}

\item Faces with vertex set $([0,n]-\{k\})\cup S'_j$ for $i<j\le\binom{n+1}{\ell+1}$.

\item Faces with vertex set $([0,n]-\{k\})\cup S'$ for a subset $S'\subseteq[n+1,2n+1]$ of size larger than $\ell+1$.

\item Faces with vertex set $S''\cup([n+1,2n+1]-\{2n+1-k\})$ for a subset $S''\subseteq[0,n]$ of size larger than $\ell$,

\end{itemize}
and $P''_{\ell,i}\subseteq\Delta^{2n+1}$ is the largest sub-simplicial set which does not have any of the following faces.

\begin{itemize}

\item Faces with vertex set $([0,n]-\{k\})\cup S'$ for a subset $S'\subseteq[n+1,2n+1]$ of size larger than $\ell+1$.

\item Faces with vertex set $S''_j\cup([n+1,2n+1]-\{2n+1-k\})$ for $i<j\le\binom{n}{\ell+1}$.

\item Faces with vertex set $S''\cup([n+1,2n+1]-\{2n+1-k\})$ for a subset $S''\subseteq[0,n]$ of size larger than $\ell+1$.

\end{itemize}
Since the inclusions
$$
P'_{\ell,i}\subset P'_{\ell,i+1}=P'_{\ell,i}\sqcup_{\Lambda^{[0,n]}_k\star\Delta^{S'_{i+1}}}(\Delta^{[0,n]}\star\Delta^{S'_{i+1}})\text{ and }
P''_{\ell,i}\subset P''_{\ell,i+1}=P''_{\ell,i}\sqcup_{\Delta^{S''_{i+1}}\star\Lambda^{[n+1,2n+1]}_{2n+1-k}}(\Delta^{S''_{i+1}}\star\Delta^{[n+1,2n+1]})
$$
are inner anodyne, the induction step is proven.

\end{proof}

\begin{lemma}\label{l:isKan}
Let ${\mathscr Y}\rightarrow T$ be a right fibration. Then $\Gamma(T,{\mathscr Y})$ is a Kan complex.
\end{lemma}

\begin{proof}
It is easy to see that ${\mathscr Y}^T\rightarrow T^T$ is a right fibration. Hence,
$\Gamma(T,{\mathscr Y})$ is a right fibration over a point. Hence a Kan complex, \cite[Lemma 2.1.3.3]{lurie2009higher}.
\end{proof}

\begin{prop}\label{prop:Map of Gamma}

Let $\mathscr X\xrightarrow pK$ be an inner fibration between quasi-categories. Let $F,G\in\Gamma(K,\mathscr X)$ be two sections. Then by Lemma \ref{lem:twisted arrow for Fun}, the induced map $\Tw K\times_{\mathscr X\times\mathscr X^\op}\Tw\mathscr X\to\Tw K\times_{K\times K^\op}\Tw K$ is a right fibration. Therefore, its pullback $Z\xrightarrow q\Tw K$ along the diagonal $\Tw K\xrightarrow{\Delta_{\Tw K}}\Tw K\times_{K\times K^\op}\Tw K$ is also a right fibration. The right fibration $q$ is classified by the diagram
$$
(\Tw K)^\op\xrightarrow{(m\xrightarrow\alpha n)\mapsto\Map_\alpha(F(m),G(n))}\mathscr S.
$$
Then the (homotopy) limit of this diagram is $\Map_{\Gamma(K,\mathscr X)}(F,G)$.

\end{prop}

\begin{proof}

First, we claim that the induced map $\HomL_{\Fun(K,\mathscr X)}(F,G)\xrightarrow{p_{F,G}^{\mathrm L}}\HomL_{\Fun(K,K)}(\id_K,\id_K)$ is a right fibration. That is, for $0<k\le n$, any lifting problem
\begin{center}

\begin{tikzpicture}[xscale=4,yscale=1.5]
\node (C') at (0,1) {$\Lambda_k^n$};
\node (D') at (1,1) {$\HomL_{\Fun(K,\mathscr X)}(F,G)$};
\node (C) at (0,0) {$\Delta^n$};
\node (D) at (1,0) {$\HomL_{\Fun(K,K)}(\id_K,\id_K)$};
\path[->,font=\scriptsize,>=angle 90]
(C') edge (D')
(C') edge (C)
(D') edge node [right] {$p_{F,G}^{\mathrm L}$} (D)
(C) edge (D);
\end{tikzpicture}

\end{center}
needs to have a solution. But this corresponds to a lifting problem
\begin{center}

\begin{tikzpicture}[xscale=4,yscale=1.5]
\node (C') at (0,1) {$(\Lambda_k^n\star\Delta^0)\cup(\Delta^n\star\emptyset)$};
\node (D') at (1,1) {$\Fun(K,\mathscr X)$};
\node (C) at (0,0) {$\Delta^{n+1}$};
\node (D) at (1,0) {$\Fun(K,K)$.};
\path[->,font=\scriptsize,>=angle 90]
(C') edge (D')
(C') edge (C)
(D') edge node [right] {$p\circ$} (D)
(C) edge (D);
\end{tikzpicture}

\end{center}
As the inclusion $(\Lambda_k^n\star\Delta^0)\cup(\Delta^n\star\emptyset)=\Lambda^{n+1}_k\subset\Delta^{n+1}$ is inner anodyne, the lifting problems have solutions.

Thus, the map $p_{F,G}^{\mathrm L}$ is a right fibration between Kan complexes and therefore a Kan fibration \cite{lurie2009higher}*{Lemma 2.1.3.3${}^\op$}. Therefore the strict Cartesian square
\begin{center}

\begin{tikzpicture}[xscale=4,yscale=1.5]
\node (C') at (0,1) {$\HomL_{\Gamma(K,\mathscr X)}(F,G)$};
\node (D') at (1,1) {$\HomL_{\Fun(K,\mathscr X)}(F,G)$};
\node (C) at (0,0) {$\{\id_{\id_K}\}$};
\node (D) at (1,0) {$\HomL_{\Fun(K,K)}(\id_K,\id_K)$.};
\path[->,font=\scriptsize,>=angle 90]
(C') edge (D')
(C') edge (C)
(D') edge node [right] {$p_{F,G}^{\mathrm L}$} (D)
(C) edge (D);
\end{tikzpicture}

\end{center}
is moreover homotopy Cartesian. We have
$$
\Map_{\Fun(K,\mathscr X)}(F,G)=\lim((\Tw K)^\op\to K^\op\times K\xrightarrow{F^\op\times G}\mathscr X^\op\times\mathscr X\xrightarrow{\Map_{\mathscr X}}\mathscr S)
$$
\cite{gepner2017lax}*{Definition 2.1, Proposition 5.1}. By construction, the presheaf $(\Tw K)^\op\to K^\op\times K\xrightarrow{F^\op\times G}\mathscr X^\op\times\mathscr X\xrightarrow{\Map_{\mathscr X}}\mathscr S$ classifies the right fibration $\Tw\mathscr X\times_{\mathscr X\times\mathscr X^\op}\Tw K\to\Tw K$ where the pullback is taken along the composite $\Tw K\to K\times K^\op\xrightarrow{F\times G^\op}\mathscr X\times\mathscr X^\op$. Therefore, we get
$$
\Map_{\Fun(K,\mathscr X)}(F,G)=\Gamma(\Tw K,\Tw\mathscr X\times_{\mathscr X\times\mathscr X^\op}\Tw K)
$$
\cite{lurie2009higher}*{Corollary 3.3.3.4${}^\op$}.
This combines to produce an equivalence
\[
\HomL_{\Fun(K,\mathscr X)}(F,G)\simeq \Gamma(\Tw K,\Tw\mathscr X\times_{\mathscr X\times\mathscr X^\op}\Tw K).
\]
We may apply the same argument with $p$ replaced with ${\rm id}_K$ to produce an equivalence
\[
\HomL_{\Fun(K,K)}(\id_K,\id_K)\simeq \Gamma(\Tw K,\Tw K\times_{K\times K^\op}\Tw K).
\]
Hence  we have a homotopy pullback diagram
\begin{center}

\begin{tikzpicture}[xscale=6,yscale=1.5]
\node (C') at (0,1) {$\Map_{\Gamma(K,\mathscr X)}(F,G)$};
\node (D') at (1,1) {$\Gamma(\Tw K,\Tw K\times_{\mathscr X\times\mathscr X^\op}\Tw\mathscr X)$};
\node (C) at (0,0) {$\{\Delta_{\Tw K}\}$};
\node (D) at (1,0) {$\Gamma(\Tw K,\Tw K\times_{K\times K^\op}\Tw K)$.};
\path[->,font=\scriptsize,>=angle 90]
(C') edge (D')
(C') edge (C)
(D') edge (D)
(C) edge (D);
\end{tikzpicture}

\end{center}
But we also have the strict pullback diagram
\begin{center}

\begin{tikzpicture}[xscale=6,yscale=1.5]
\node (C') at (0,1) {$\Gamma(\Tw K,Z)$};
\node (D') at (1,1) {$\Gamma(\Tw K,\Tw K\times_{\mathscr X\times\mathscr X^\op}\Tw\mathscr X)$};
\node (C) at (0,0) {$\{\Delta_{\Tw K}\}$};
\node (D) at (1,0) {$\Gamma(\Tw K,\Tw K\times_{K\times K^\op}\Tw K)$.};
\path[->,font=\scriptsize,>=angle 90]
(C') edge (D')
(C') edge (C)
(D') edge (D)
(C) edge (D);
\end{tikzpicture}

\end{center}
It sufficess to show that the map
$$
\Gamma(\Tw K,\Tw K\times_{\mathscr X\times\mathscr X^\op}\Tw\mathscr X)\to\Gamma(\Tw K,\Tw K\times_{K\times K^\op}\Tw K)
$$
is a right fibration. This is because the section categories are Kan complexes by the previous lemma
and \cite[2.1.3.3]{lurie2009higher}.

The fact that it is a right fibration follow from follows from Lemma \ref{lem:twisted arrow for Fun} and that for $0<k\le n$, the inclusion $\Lambda^n_k\times\Tw K\to\Delta^n\times\Tw K$ is right anodyne \cite{lurie2009higher}*{Corollary 2.1.2.7}.

\end{proof}

\begin{cor}\label{cor:contractible fibrewise mapping spaces}

Let $\mathscr X\xrightarrow pK$ be an inner fibration between quasi-categories. Let $F,G\in\Gamma_{\Cart}(K,\mathscr X)$ be two Cartesian sections. Suppose that $\Map_{\mathscr X_k}(F(k),G(k))$ is contractible for all $k\in K$. Then $\Map_{\Gamma(K,\mathscr X)}(F,G)$ is contractible too.

\end{cor}

\begin{proof}

Let $k\xrightarrow e\ell$ be an edge of $K$. By assumption, $G_e$ is a $p$-Cartesian edge. Therefore, postcomposition with it gives an equivalence $\Map_k(F(k),G(k))\to\Map_e(F(k),G(\ell))$. Then by the formula of Proposition \ref{prop:Map of Gamma}, $\Map_{\Gamma(K,\mathscr X)}(F,G)$ is a limit of contractible spaces, thus contractible itself.

\end{proof}

\subsection{Mapping spaces in the Cartesian section quasi-category}

Let $X\to S$ be a coCartesian fibration. In this subsection, $X^\natural$ will denote the marked simplicial set with marked edges the coCartesian edges.

\begin{lem}\label{lem:equivalence iff initial in quasi-category}

Let $k\xrightarrow e\ell$ be an edge in a quasi-category $K$. Then $e$ is an equivalence if and only if it is an initial object of $K_{k/}$.

\end{lem}

\begin{proof}
An $n$-simplex of $(K_{k/})_{e/}$ is by definition a morphism
$\Delta^{n+2}=\Delta^0\star\Delta^0\star\Delta^n\rightarrow K$ whose restriction to $\Delta^1= \Delta^0\star \Delta^0$
is $e$.
Hence we have a commutative triangle
\begin{center}

\begin{tikzpicture}[xscale=2,yscale=1.5]
\node (C') at (0,1) {$(K_{k/})_{e/}$};
\node (D') at (1,1) {$K_{e/}$};
\node (C) at (0.5,0) {$K_{k/},$};
\path[->,font=\scriptsize,>=angle 90]
(C') edge node [above] {$\cong$} (D')
(C') edge node [left] {$r$} (C)
(D') edge node [right] {$s$} (C);
\end{tikzpicture}

\end{center}
The vertical maps are obtained by restricting to $\Delta^{[n+2]\setminus\{1\}}$.
  The map $e$ is an initial object of $K_{k/}$ if and only if it is a strongly initial object \cite{lurie2009higher}*{Corollary 1.2.12.5}, that is the map $r$ is a trivial fibration. Unwinding the definitions, the map $s$ has the right lifting property with respect to $\partial\Delta^n\subset\Delta^n$ if and only if every diagram $\Lambda^{n+2}_0\xrightarrow\sigma K$ with $\sigma|\Delta^{[0,1]}=e$ extends to $\Delta^{n+2}$. The latter lifting property is equivalent to $e$ being an equivalence by \cite{lurie2009higher}*{Proposition 1.2.4.3}.

\end{proof}

\begin{lem}\label{lem:fib of sharp is Kan}

Let $X$ be a quasi-category. Let $X^\sharp\xrightarrow q Y^\natural$ be a coCartesian trivial cofibration in $\Set_\Delta^+$. Then $Y$ is a Kan complex.

\end{lem}

\begin{proof}
It is automatic that $Y$ is a quasi-category hence it suffices to show that all its edges are equvialences.
First note that the coCartesian edges in $Y$ (over a point), are exactly the equivalences.
It follows that $\Map^\sharp(X^\sharp,(Y^\simeq)^\sharp)=\Map^\sharp(X^\sharp,Y^\natural)$. Therefore, since $q\in\Map^\sharp(X^\sharp,Y^\natural)=\Map^\sharp(X^\sharp,(Y^\simeq)^\sharp)$, and the precomposition map $\Map^\sharp(Y^\natural,(Y^\simeq)^\sharp)\xrightarrow{\circ q}\Map^\sharp(X^\sharp,(Y^\simeq)^\sharp)$ is a homotopy equivalence of Kan complexes, there exists $Y\xrightarrow fY^\simeq$ together with a homotopy $fq\sim q$. Since the precomposition map $\Map^\sharp(Y^\natural,Y)\xrightarrow{\circ q}\Map^\sharp(X^\sharp,Y^\natural)$ is a homotopy equivalence, this implies that there exists a homotopy $f\xrightarrow H\id_Y$ in $\Map^\sharp(Y^\natural,Y^\natural)$. But then for all edges $x\xrightarrow ey$ in $Y$, we get a homotopy commutative square in $Y$:
\begin{center}

\begin{tikzpicture}[xscale=2,yscale=1.5]
\node (C') at (0,1) {$x$};
\node (D') at (1,1) {$f(x)$};
\node (C) at (0,0) {$y$};
\node (D) at (1,0) {$f(y)$};
\path[->,font=\scriptsize,>=angle 90]
(C') edge node [above] {$H_x$} node[below] {$\simeq$} (D')
(C') edge node [right] {$e$} (C)
(D') edge node [right] {$f(e)$} node[left] {$\simeq$} (D)
(C) edge node [above] {$H_y$} node[below] {$\simeq$} (D);
\end{tikzpicture}

\end{center}
showing that $e$ is an equivalence.

\end{proof}

\begin{lem}\label{lem:undercat of Kan is contractible}

Let $y\in Y$ be a vertex in a Kan complex. Then the undercategory $Y_{y/}$ is contractible.

\end{lem}

\begin{proof}

Consider a lifting problem
\begin{center}

\begin{tikzpicture}[xscale=2,yscale=1.5]
\node (C') at (0,1) {$\partial\Delta^n$};
\node (D') at (1,1) {$Y_{y/}$};
\node (C) at (0,0) {$\Delta^n.$};
\path[->,font=\scriptsize,>=angle 90]
(C') edge (D')
(C') edge (C)
(C) edge [dashed] (D');
\end{tikzpicture}

\end{center}
It corresponds to a lifting problem
\begin{center}

\begin{tikzpicture}[xscale=2,yscale=1.5]
\node (C') at (0,1) {$\Lambda^{n+1}_0$};
\node (D') at (1,1) {$Y$};
\node (C) at (0,0) {$\Delta^{n+1},$};
\path[->,font=\scriptsize,>=angle 90]
(C') edge (D')
(C') edge (C)
(C) edge [dashed] (D');
\end{tikzpicture}

\end{center}
which has a solution as $Y$ is a Kan complex.

\end{proof}

\begin{lem}\label{lem:initial object in simplicial localization}

Let $X$ be a quasi-category. Suppose that it has an initial object $x\in X$. Let $X^\sharp\xrightarrow qY^\natural$ be a coCartesian trivial cofibration in $\Set_\Delta^+$. Then $q(x)\in Y$ is an initial object.

\end{lem}

\begin{proof}

Since $x\in X$ is initial, it is strongly initial \cite{lurie2009higher}*{Corollary 1.2.15.5}, that is the restriction map $X_{x/}\xrightarrow{r_X}X$ is a trivial fibration. Therefore, it has a section $X\xrightarrow{s_X}X_{x/}$ such that $\id_{X_{x/}}\sim s_Xr_X$. Since the quasi-category $Y$ is a Kan complex by Lemma \ref{lem:fib of sharp is Kan}, so is the undercategory $Y_{q(x)/}$. Therefore, the induced map $X_{x/}\xrightarrow{q_{x/}}Y_{q(x)/}$ takes all edges into equivalences, and therefore it induces a map $(X_{x/})^\sharp\xrightarrow{q_{x/}}(Y_{q(x)/})^\natural$. Therefore, as the postcomposition map $\Map^\sharp(Y^\natural,(Y_{q(x)/})^\natural)\xrightarrow{\circ q}\Map^\sharp(X^\sharp,(Y_{q(x)/})^\natural)$ is a homotopy equivalence, there exists a map $Y\xrightarrow{s_Y}Y_{q(x)/}$ and a homotopy $s_Yq\sim q_{x/}(s_X)^\sharp$. Since we have
$$
r_Ys_Yq\sim r_Y q_{x/}(s_X)^\sharp=q(r_X)^\sharp(s_X)^\sharp=q,
$$
we get a homotopy $r_Ys_Y\sim\id_Y$. Moreover, as the space $Y_{q(x)/}$ is contractible by Lemma \ref{lem:undercat of Kan is contractible}, the canonical map $Y_{q(x)/}\xrightarrow p\ast$ has a section $\ast\xrightarrow iY_{q(x)/}$ together with a homotopy $\id_{Y_{q(x)/}}\sim ip$. Therefore, we have a homotopy
$$
s_Yr_Y\sim ips_Yr_Y=ip\sim\id_{Y_{q(x)/}}.
$$
Therefore, $r_Y$ is a homotopy equivalence, and thus it is a weak equivalence. Since it is moreover a left fibration, it is a trivial fibration. This shows that $q(x)\in Y$ is strongly initial, and thus it is an initial object.

\end{proof}

\begin{lem}\label{lem:nu loc}

Let $K$ be a quasi-category. Let $\Tilde\nu$ denote the collection of edges of $(\Tw K)^\op$ given by diagrams in $K$ of the form
\begin{center}

\begin{tikzpicture}[xscale=2,yscale=1.5]
\node (C') at (0,1) {$m$};
\node (D') at (1,1) {$m$};
\node (C) at (0,0) {$n'$};
\node (D) at (1,0) {$n$.};
\path[->,font=\scriptsize,>=angle 90]
(C') edge node [above] {$\id_m$} (D')
(C') edge node [right] {$\alpha'$} (C)
(D') edge node [right] {$\alpha$} (D)
(D) edge node [above] {$\nu$} (C);
\end{tikzpicture}

\end{center}
Let $((\Tw K)^\op,\Tilde\nu)\xrightarrow qH$ be a coCartesian trivial cofibration in $(\Set_\Delta^+)_{/K^\op}$ with source the restriction map $((\Tw K)^\op,\Tilde\nu)\xrightarrow r (K^\op)^\sharp$ and target a coCartesian fibration $H^\natural\xrightarrow p(K^\op)^\sharp$. Then $H^\natural\xrightarrow p (K^\op)^\sharp$ is a trivial coCartesian fibration.

\end{lem}

\begin{proof}

It will be enough to show that the fibres of $H\xrightarrow pK^\op$ are contractible \cite{lurie2009higher}*{Proposition 3.1.3.5}. Fix $m\in K^\op$. Note that the fibre $r^{-1}(m)=((\Tw K)^\op,\Tilde\nu)_m=(K_{m/})^\sharp$ by construction. The fibre $(K_{m/})^\sharp\xrightarrow{q_m}H_m$ is a trivial coCartesian cofibration in $\Set_\Delta^+$. By Lemma \ref{lem:equivalence iff initial in quasi-category}, $\id_m\in K_{m/}$ is an initial object. Then by Lemma \ref{lem:initial object in simplicial localization}, $H_m$ has an initial object $y\in H_m$. That is, the restriction map $(H_m)_{y/}\to H_m$ is a trivial fibration. But by Lemma \ref{lem:fib of sharp is Kan}, $H_m$ is a Kan complex. Therefore, so is $(H_m)_{y/}$. Then by Lemma \ref{lem:undercat of Kan is contractible}, $(H_m)_{y/}$ is contractible. Therefore, $H_m$ is contractible as claimed.

\end{proof}

\begin{lem}\label{lem:Cart factors through nu}

Let $\mathscr X\xrightarrow pK$ be an inner fibration over a quasi-category. Let $F,G\in\Gamma(K,\mathscr X)$ be two sections. Suppose that $G$ is a Cartesian section. Then the map $(\Tw K)^\op\xrightarrow{f:\alpha\mapsto\Map_\alpha(F_m,G_n)}\mathscr S$ takes a diagram $\nu'\in\Tilde\nu$:
\begin{center}

\begin{tikzpicture}[xscale=2,yscale=1.5]
\node (C') at (0,1) {$m$};
\node (D') at (1,1) {$m$};
\node (C) at (0,0) {$n'$};
\node (D) at (1,0) {$n$.};
\path[->,font=\scriptsize,>=angle 90]
(C') edge node [above] {$\id_m$} (D')
(C') edge node [right] {$\alpha'$} (C)
(D') edge node [right] {$\alpha$} (D)
(D) edge node [above] {$\nu$} (C);
\end{tikzpicture}

\end{center}
to an equivalence.

\end{lem}

\begin{proof}

As shown in the Proof of Proposition \ref{prop:Map of Gamma}, the map $f$ classifies the right fibration $Z\to\Tw K$ that is the pullback of $\Tw K\times_{\mathscr X\times\mathscr X^\op}\Tw\mathscr X\xrightarrow q\Tw K\times_{K\times K^\op}\Tw K$ along $\Tw K\xrightarrow\Delta\Tw K\times_{K\times K^\op}\Tw K$, where $q$ itself is the pullback of $\Tw\mathscr X\xrightarrow\lambda\mathscr X\times\mathscr X^\op$ along $\Tw K\xrightarrow\lambda K\times K^\op\xrightarrow{(F,G^\op)}\mathscr X\times\mathscr X^\op$. Since $\lambda$ corresponds to $\mathscr X^\op\xrightarrow{h^\bullet}\Fun(\mathscr X,\mathscr S)$ \cite{lurie2014higher}*{Proposition 5.2.1.11}, the map $(\Tw K)^\op\xrightarrow{\alpha\mapsto\Map_{\mathscr X}(F_m,G_n)}\mathscr S$ takes $\nu'$ to the postcomposition map $\Map_{\mathscr X}(F_m,G_n)\xrightarrow{G_\nu\circ}\Map_{\mathscr X}(F_m,G_{n'})$. Therefore, its restriction $f$ takes $\nu'$ to the postcomposition map $\Map_{\alpha}(F_m,G_n)\xrightarrow{G_\nu\circ}\Map_{\alpha\circ\nu}(F_m,G_{n'})$, which is an equivalence as $G$ is a Cartesian fibration.

\end{proof}

\begin{prop}\label{prop:Map of Gamma_Cart}

Let $\mathscr X\xrightarrow p K$ be an inner fibration over a quasi-category. Let $F,G\in\Gamma_{\Cart}(K,\mathscr X)$ be two Cartesian sections. Then we have
$$
\Map_{\Gamma_{\Cart}(K,\mathscr X)}(F,G)\simeq\lim(K^\op\xrightarrow{m\mapsto\Map_m(F_m,G_m)}\mathscr S).
$$

\end{prop}

\begin{proof}

Consider the diagram

\begin{center}

\begin{tikzpicture}[xscale=3,yscale=3]
\node (C') at (0,0) {$(\Tw K)^\op$};
\node (D') at (1,0) {$((\Tw K)^\op,\Tilde\nu)$};
\node (E') at (2,0) {$\mathscr S$};
\node (C) at (.5,-.5) {$H$};
\node (D) at (1,-1) {$K^\op.$};
\node (E) at (2,-1) {$\mathscr S\times K^\op$};
\path[->,font=\scriptsize,>=angle 90]
(C') edge node [above] {$i$} (D')
(D') edge node [above] {$f'$} (E')
(C') edge [bend left] node [above] {$f\colon(m\xrightarrow{\alpha}n)\mapsto\Map_{\alpha}(F_m,G_n)$} (E')
(D') edge node [left] {$q$} (C)
(C') edge node [below] {$qi$} (C)
(D) edge [bend left] node [left] {$t$} (C)
(D') edge node [right] {$(f',r)$} (E)
(C) edge node [above] {$(g',p)$} (E)
(D) edge node [above] {$(g,\id)$} (E)
(E) edge node [right] {$\pr$} (E')
(C) edge node [right] {$p$} (D);
\end{tikzpicture}

\end{center}
Here, $i$ is the map given by the identity, and $q$ is a coCartesian trivial cofibration in $(\Set_\Delta^+)_{/K^\op}$ between the restriction map $((\Tw K)^\op,\Tilde\nu)\xrightarrow r(K^\op)^\sharp$ and a coCartesian fibration $H\xrightarrow p(K^\op)^\sharp$. Since by Lemma \ref{lem:Cart factors through nu}, $f$ takes the $\nu'\in\Tilde\nu$ to equivalences, we have $f=if'$. Since $p$ is a coCartesian trivial fibration by Lemma \ref{lem:nu loc}, it has a section $t$. This shows that $t$ is co-marked anodyne, and thus left anodyne. We claim that $qi$ is also left anodyne. Let $X\xrightarrow {\pi}S$ be a left fibration. Then it is a coCartesian fibration. Therefore, $X^\sharp=X^\natural\xrightarrow{\pi^\sharp} S^\sharp$ is a coCartesian fibration. Consider a lifting diagram
\begin{center}

\begin{tikzpicture}[xscale=2,yscale=3]
\node (C') at (0,1) {$(\Tw K)^\op$};
\node (D') at (1,1) {$X^\sharp$};
\node (E') at (0,.5) {$(\Tw K)^\op[\Tilde\nu]$};
\node (C) at (0,0) {$H$};
\node (D) at (1,0) {$S^\sharp$.};
\path[->,font=\scriptsize,>=angle 90]
(C') edge node [above] {$a$} (D')
(E') edge [dashed] node [above] {$a'$} (D')
(E') edge node [right] {$q$} (C)
(C') edge node [right] {$i$} (E')
(D') edge node [right] {$\pi^\sharp$} (D)
(C) edge [dashed] node [right] {$c$} (D')
(C) edge node [above] {$b$} (D);
\end{tikzpicture}

\end{center}
Since every edge of $X^\natural$ is marked, $a$ gives $a'$. Since $q$ is co-marked anodyne, we get $c$ such that $cq=a'$ and $\pi^\sharp c=b$. But then $cqi=a$ and $\pi^\sharp c=b$ shows that $qi$ has the left lifting property with respect to $\pi$.

Since the precomposition map $\Map^\sharp(H^\natural,(\mathscr S\times K^\op)^\natural)\xrightarrow{\circ q}\Map^\sharp(((\Tw K)^\op,\Tilde\nu),(\mathscr S\times K^\op)^\natural)$ is a homotopy equivalence, there exists a map $H\xrightarrow{(g',p)}\mathscr S\times K^\op$ together with a homotopy $(g',p)q\sim(f',r)$. Let $g=g't$. We know that a map is left anodyne if and only if it is final \cite{lurie2009higher}*{Proposition 4.1.1.3 (4)}. That is, the maps $qi$ and $t$ are final. Then, using Proposition \ref{prop:Map of Gamma}, we get
$$
\Map_{\Gamma(K,\mathscr X)}(F,G)=\lim(f)=\lim(f'i)=\lim(g'qi)=\lim(g')=\lim(g).
$$

\end{proof}

\begin{rem}

Let $m\in K^\op$. Then the homotopy $(f',\pr)\sim(g',p)q$ gives a homotopy $f'\sim g'q$. Moreover, as $t$ is a section of the coCartesian trivial fibration $p$, we get a homotopy $pt\sim\id_H$. Therefore, from $m=r(\id_m)=p(q(\id_m))$, we get a homotopy $t(m)\sim q(\id_m)$. This gives
$$
g(m)=g'(t(m))\simeq g'(q(\id_m))\simeq f'(\id_m)=\Map_m(F_m,G_m).
$$

\end{rem}


\section{The construction of ${}^\op\mathscr D_S$} \label{s:fibration}

In this section, we relativize the dg-nerve construction \cite{lurie2014higher}*{Construction 1.3.1.6} to give a presentable fibration over $\Sch_S$, which is classified by the functor $\Sch_S^\op\xrightarrow{T\mapsto\mathscr D(T)^\op}\Cat_\infty$.

\begin{cons}\label{cons:fibration}

Let ${}^\op\mathscr D_S$ denote the simplicial set with $n$-simplices tuples $(\sigma,(K_i)_{i\in[n]},(f_I)_{I\subseteq[n]})$, where

\begin{enumerate}

\item $\sigma=(T_i,t_{ij})$ is an $n$-simplex in the nerve of the category $\Sch_S$. That is, for each $0\le i\le n$, $T_i$ is an $S$-scheme, and for each $0\le i<j\le n$, $t_{ij}$ is a morphism of $S$-schemes $T_i\to T_j$ such that for each $0\le i<j<k\le n$, we have $t_{ij}=t_{jk}t_{ij}$.

\item For each $0\le i\le n$, $K_i$ is a complex of injective $\mathscr O_{T_i}$-modules with quasi-coherent cohomology sheaves.

\item For each $I=\{i_-<i_1<\dotsb<i_m<i_+\}\subseteq[n]$ with $m\ge0$, we have $f_I\in\RHom_{T_{i_-}}^m(Lt_{i_-i_+}^*K_{i_+},K_{i_-})$ such that
$$
df_I=\sum_{1\le j\le m}(-1)^j(f_{I-\{i_j\}}\circ(t_{i_-i_+}^*r_{Lt_{i_ji_+}^*K_{i_+}})-f_{\{i_-,\dotsc,i_j\}}\circ(Lt_{i_-i_j}^*f_{\{i_j,\dotsc,i_+\}})).
$$

\end{enumerate}

Let $[m]\xrightarrow\alpha[n]$ be a morphism in $\Delta$. Then the corresponding map $({}^\op\mathscr D_S)_n\xrightarrow{\alpha^*}({}^\op\mathscr D_S)_m$ is defined as
$$
(\sigma,(K_i)_{i\in[n]},(f_I)_{I\subseteq[n]})\mapsto(\sigma\circ\alpha,(K_{\alpha(j)})_{j\in[m]},(g_J)_{J\subseteq[m]}),
$$
where
$$
g_J=\begin{cases}
f_{\alpha(J)} & \alpha|J\text{ is injective,}\\
\id_{I_i} & |J|=2\text{ and }\alpha(J)=\{i\},\\
0 & \text{else.}
\end{cases}
$$
We need to check that this makes sense, in other words the third condition
in the constructions holds. To see this, observe that any $\alpha$ can be factored
into coface and codegeneracy maps, hence it suffices to check the condition for
those. In the case, of a coface map, we will always be in the situation where
$\alpha|J$ is injective, hence the condition for $g_J$ boils down to the same condition
for $f_J$.

Now consider the case where $\alpha : [n+1]\rightarrow [n]$ is a degeneracy
with $\alpha(i)=\alpha(i+1)$. We may assume that both $i,i+1\in J$ otherwise
we will be in the injective situation. Then $g_J=0$. In the sum, all terms of the form
\[
g_{\{i_-,\dotsc,i_j\}}\circ(Lt_{i_-i_j}^*g_{\{i_j,\dotsc,i_+\}})
\]
will vanish due to non-injectivity of the restricted morphism.
Exactly two of the terms $g_{J-{i_j}}$ will not vanish but will occur with opposite sign.

Let ${}^\op\mathscr D_S\xrightarrow p\Sch_S$ denote the forgetful map.

\end{cons}

\begin{rem}

We put the op in the notation to avoid confusion as the fibres ${}^\op\mathscr D(T)$ are indeed the opposite categories of the derived category $\mathscr D(T)$. To see this, note that the dg-nerve functor \cite{lurie2014higher}*{Proposition 1.3.1.20} commutes with opposites. We will write $\mathscr D_S=({}^\op\mathscr D_S)^\op$.

\end{rem}

\begin{notn}\label{n:truncation}
	For each integer $n$ we can consider
	 full subcategories ${}^\op\msD^{\le n}_S$ (resp. ${}^\op\msD^{\ge n}_S$) of
	 ${}^\op\msD_S$
	 on complexes with cohomology in degrees at most (resp. at least) $n$.
\end{notn}

Given a simplicial set $S$ and a pair of $0$-simplicies $s,t\in S_0$ recall the
definition of the right mapping space $\Hom^R(s,t)$ from
\cite[page 27]{lurie2009higher}

\begin{prop}\label{prop:Map in D}

Let $(U,J),(T,I)\in{}^\op\mathscr D_S$ be $0$-simplices. Then we have $$\Hom^R_{{}^\op\mathscr D_S}((U,J),(T,I))\cong\bigsqcup_{U\xrightarrow gT}\DK\tau_{\le0}\RHom_U(Lg^*I,J).$$

\end{prop}

\begin{proof}

This is a direct generalization of \cite{lurie2014higher}*{Remark 1.3.1.12}. Let $\Delta^{n+1}\xrightarrow\lambda{}^\op\mathscr D_S$ be an $n$-simplex of $\Hom^R_{{}^\op\mathscr D_S}((U,J),(T,I))$.

Given a pair of integers $0\le i<j\le n+1$ there is an induced morphism
 $\delta_{ij}:\Delta^1\hookrightarrow\Delta^{n+1}$ induced by the inclusion
$[1]\hookrightarrow [n+1]$ with image $\{i,j\}$.
 Let $g=p\circ \lambda \circ \delta_{n,n+1}$, a morphism of schemes. Then by construction, for $0\le i<j\le n+1$, we have
$$
p \circ \lambda \circ \delta_{ij}=\begin{cases}
\id_U & j<n+1\\
g & j=n+1.
\end{cases}
$$
This shows that by construction we have
\begin{eqnarray*}
\Hom_{\Set_\Delta}(\Delta^n,\Hom^R_{{}^\op\mathscr D_S}((U,J),(T,I)))
&\cong&\bigsqcup_{U\xrightarrow gT}\Hom_{\Ch(\mathbf Z)}(N\mathbf Z\Delta^n,\RHom_U(Lg^*I,J))\\
&=&\bigsqcup_{U\xrightarrow gT}\Hom_{\Ab_\Delta}(\mathbf Z\Delta^n,\DK\tau_{\le0}\RHom_U(Lg^*I,J)) \\
&\cong&\bigsqcup_{U\xrightarrow gT}\Hom_{\Set_\Delta}(\Delta^n,\DK\tau_{\le0}\RHom_U(Lg^*I,J)).
\end{eqnarray*}

\end{proof}

\begin{notn}

Let $U\xrightarrow gT$ be a morphism of schemes, $I\in\mathscr D(T)$, and $J\in\mathscr D(U)$. Then we let $\Hom^R_{{}^\op\mathscr D_S,g}(J,I)=\Hom^R_{{}^\op\mathscr D_S}(J,I)\times_{\Hom^R_{\Sch_S}(U,T)}\{g\}$.

\end{notn}

If $\msC$ is an $\infty$-category, the largest sub-Kan complex of $\msC$, denoted $\msC^\simeq$ is
called \emph{the interior of} $\msC$. If $\msC$ is the nerve of an ordinary category, then its interior
is the nerve of the largest subgroupoid in $\msC$.
 In the case of a Cartesian fibration $\msD\rightarrow N$ we may take the
 subcategory on Cartesian edges $\msD_{\Cart}\rightarrow N$ and obtain a right fibration. The fibers
 are now Kan complexes, see \cite[2.1.3]{lurie2009higher}.

\begin{cor}

Let $X\xrightarrow f S$ be a flat morphism of schemes. Then by construction, the 1-category $\mathscr D^\flat_{\mathrm{pug}}(X/S)$ \cite{lieblich2006moduli}*{\S2.1} is equivalent to the 1-truncation of the interior of the full subcategory of $f_*\mathscr D_{X}$ on universally gluable $S$-perfect complexes.

\end{cor}

We will call a map $f:X\rightarrow Y$ of simplicial sets a \emph{presentable fibration} if
it is an inner fibration that is both Cartesian and coCartesian, see \cite[Ch. 2]{lurie2009higher}, and its
fibres are presentable quasi-categories.

\begin{lem}\label{lem:p is inner fibration}

The map ${}^\op\mathscr D_S\xrightarrow p\Sch_S$ is an inner fibration.

\end{lem}

\begin{proof}

For $n\ge2$ and $0<k<n$, consider a lifting problem
\begin{center}

\begin{tikzpicture}[xscale=4,yscale=1.5]
\node (C') at (0,1) {$\Lambda^n_k$};
\node (D') at (1,1) {${}^\op\mathscr D_S$};
\node (C) at (0,0) {$\Delta^n$};
\node (D) at (1,0) {$\Sch_S$.};
\path[->,font=\scriptsize,>=angle 90]
(C') edge node [above] {$(\sigma|_{\Lambda^n_k},(K_i)_{i\in[n]},(f_I)_{\Delta^I\subseteq\Lambda^n_i})$} (D')
(C') edge (C)
(D') edge (D)
(C) edge node [below] {$\sigma=((T_i)_{i\in[n]},(f_{ij})_{0\le i<j\le n})$} (D)
(C) edge [dashed] (D');
\end{tikzpicture}

\end{center}

Note that an $n$-simplex of $\Sch_S$ amounts to a sequence of schemes and maps
\[
T_0 \xrightarrow{f_{01}} T_1\xrightarrow{f_{12}} T_2 \xrightarrow{f_{23}}\cdots
\rightarrow T_n.
\]
The dotted arrow in the diagram amounts to giving the data of $f_I$ for
$I=\{i_-<i_1<\ldots < i_+\}$ as in (\ref{cons:fibration}). There are two possibilities,
either $i_-=0$ or otherwise. When $i_-\ne 0$ then the data of $f_I$ is determined
by the factorisation $\Delta^{\{1,2,\ldots, n\}}\subseteq \Lambda_k^n \subseteq \Delta^n$
as the horn is inner.

The upshot is that
the original lifting problem reduces to problem of the form

\begin{center}

\begin{tikzpicture}[xscale=5,yscale=1.5]
\node (C') at (0,1) {$\Lambda^n_k$};
\node (D') at (1,1) {$N_\dg(\Ch(T_0))$.};
\node (C) at (0,0) {$\Delta^n,$};
\path[->,font=\scriptsize,>=angle 90]
(C') edge node [above] {$((Lf_{0i}^*I_i)_{i\in[n]},(Lf_{0i_-}^*f_I)_{\Delta^I\subseteq\Lambda^n_i})$} (D')
(C') edge (C)
(C) edge [dashed] (D');
\end{tikzpicture}

\end{center}
Such a problem  has a solution  by \cite{lurie2014higher}*{Proposition 1.3.1.10}.

\end{proof}

\begin{lemma}\label{l:cartEdge}\label{lem:Cartesian edges}
Let $U\xrightarrow gT$ be a morphism of $S$-schemes, and $I$ a complex of injective $\mathscr O_T$-modules. Then an edge in ${}^\op\msD_S$ corresponding
corresponding to a morphism $Lg^*I\xrightarrow qJ$ is Cartesian
precisely when $\tau_{\le0}q$ is a quasi-isomorphism.
\end{lemma}

\begin{proof}

Let $V\xrightarrow hU$ be a morphism of schemes and $K\in\mathscr D(V)$. Then as the precomposition map $\RHom_V(Lh^*Lg^*I,K)\xleftarrow{\circ h^*r_{Lg^*I}}\RHom_V(L(gh)^*I,K)$ is a quasi-isomorphism between dg-injective complexes, it has a section $s$. Let us define a section $t$ of the restriction map
$$
\Hom^R_{{}^\op\mathscr D_S,h}(K,J)=({}^\op\mathscr D_S)_{/J}\times_{{}^\op\mathscr D_S\times(\Sch_S)_{/U}}\{(K,h)\}\leftarrow({}^\op\mathscr D_S)_{/q}\times_{{}^\op\mathscr D_S\times(\Sch_S)_{/g}}\{(K,(V\xrightarrow hU\xrightarrow gT))\}
$$
by mapping $\Delta^n\to\Hom^R_{{}^\op\mathscr D_S,h}(K,J)$ defined by $\Delta^{n+1}\xrightarrow{(\sigma,(K_i),(f_I))}{}^\op\mathscr D_S$ to $\Delta^{n+2}\xrightarrow{(\Bar\sigma,(\Bar K_i),(\Bar f_I))}{}^\op\mathscr D_S$, which is defined as follows.

\begin{enumerate}

\item We have $\Bar\sigma=(V\xrightarrow\id V\to\dotsb\to V\xrightarrow hU\xrightarrow gT)$.

\item We have $\Bar K_0=\dotsb=\Bar K_n=K$, $\Bar K_{n+1}=J$, and $\Bar K_{n+2}=I$.

\item

\begin{enumerate}

\item If $I\subseteq[0,n+1]$, then $\Bar f_I=f_I$.

\item We have $\Bar f_{\{n+1,n+2\}}=q$.

\item If $[n,n+2]\subseteq I$, then $\Bar f_I=0$.

\item If $\{n+2\}\subset I\subseteq[0,n]\cup\{n+2\}$, then $\Bar f=s(f_{I\cup\{n+1\}\setminus\{n+2\}}\circ Lh^*q)$.

\end{enumerate}

\end{enumerate}

Then we get a commutative diagram
\begin{center}

\begin{tikzpicture}[xscale=4,yscale=1.5]
\node (C') at (-1,1) {$\Map_V(Lh^*J,K)$};
\node (E') at (0.5,1) {$\Map_V(Lh^*Lg^*I,K)$};
\node (D') at (2,1) {$\Map_V(L(gh)^*I,K)$};
\node (C) at (-1,0) {$\Hom^R_{{}^\op\mathscr D_S,h}(K,J)$};
\node (E) at (0.5,0) {$({}^\op\mathscr D_S)_{/q}\times_{{}^\op\mathscr D_S\times(\Sch_S)_{/g}}\{(K,(V\xrightarrow hU\xrightarrow gT))\}$};
\node (D) at (2,0) {$\Hom_{{}^\op\mathscr D_S,gh}^R(K,I)$};
\path[->,font=\scriptsize,>=angle 90]
(C') edge node [above] {$\circ Lh^*q$} (E')
(E') edge node [above] {$\DK\tau_{\le0}s$} node [below] {$\simeq$} (D')
(C') edge node [right] {$\cong$} (C)
(D') edge node [right] {$\cong$} (D)
(C) edge node [above] {$t$} node [below] {$\simeq$} (E)
(E) edge node [above] {$\res$} (D);
\end{tikzpicture}

\end{center}
where the vertical maps are components of the isomorphisms of Proposition \ref{prop:Map in D}.

If $\tau_{\le0}q$ is a quasi-isomorphism, then so is $\RHom_V(Lh^*J,K)\xrightarrow{\circ Lh^*q}\RHom_V(Lh^*Lg^*I,K)$. Thus the map $\Map_V(Lh^*J,K)\xrightarrow{\circ\DK\tau^{\le0}Lh^*q}\Map_V(Lh^*Lg^*I,K)$ is an equivalence \cite{schwede2003equivalences}*{\S4.1}. This in turn implies that the restriction map $({}^\op\mathscr D_S)_{/q}\times_{{}^\op\mathscr D_S\times(\Sch_S)_{/g}}\{(K,(V\xrightarrow hU\xrightarrow gT))\}\to\Hom_{{}^\op\mathscr D_S,gh}^R(K,I)$ is an equivalence too. This shows that $q$ gives a $p$-Cartesian edge of ${}^\op\mathscr D_S$ \cite[Proposition 2.4.4.3]{lurie2009higher}.

On the other hand, suppose that $q$ gives a $p$-Cartesian edge of ${}^\op\mathscr D_S$. Then for all $K\in\mathscr D(U)$, the map $\Map_V(J,K)\xrightarrow{\circ\DK\tau^{\le0}q}\Map_V(Lg^*I,K)$ is an equivalence. Therefore, $\DK\tau_{\le0}q$ is an equivalence. This implies that $\tau_{\le0}q$ is a quasi-isomorphism.

\end{proof}

Let $U\xrightarrow gT$ be a morphism of $S$-schemes. Since ${}^\op\mathscr D_S\xrightarrow p\Sch_S$ is a Cartesian fibration, its restriction ${}^\op\mathscr D_S|g\to\Delta^1$ is classified by a functor of quasi-categories $\mathscr D(T)^\op\xrightarrow{g^*_{{}^\op\mathscr D_S}}\mathscr D(U)^\op$. We want to show that this functor is equivalent to the opposite $\mathscr D(T)^\op\xrightarrow{(Lg^*)^\op}\mathscr D(U)^\op$ of the derived pullback functor we have constructed in Notation \ref{notn:derived pullback and pushforward}. We will show this by showing that the Cartesian fibration $p|g$ is equivalent to the opposite of the \emph{relative nerve} of $(Lg^*)^\op$. That is a Cartesian fibration over $\Delta^1$, which is classified by $g^*_{{}^\op\mathscr D_S}$ \cite{lurie2009higher}*{Corollary 3.2.5.20}.

Let $\mathbf C$ be a 1-category. Let's recall the construction of the opposite $N_F(\mathbf C)^\op$ of the relative nerve of a functor $\mathbf C^\op\xrightarrow F\Set_\Delta$. For a finite linearly ordered set $J$, a simplex $\Delta^J\to N_F(\mathbf C)^\op$ is given by the following data.

\begin{enumerate}

\item A diagram $J\xrightarrow\sigma\mathbf C$,

\item and for each $\emptyset\ne J'\subseteq J$ with $\min J'=j'$, a map of simplicial sets $\Delta^{J'}\xrightarrow{\tau(J')}F(\sigma(j'))$, such that

\item for each $\emptyset\ne J''\subseteq J'\subseteq J$ with $\min J'=j'$ and $\min J''=j''$, the diagram
\begin{center}

\begin{tikzpicture}[xscale=4,yscale=1.5]
\node (C') at (0,1) {$\Delta^{J''}$};
\node (D') at (1,1) {$F(\sigma(j''))$};
\node (C) at (0,0) {$\Delta^{J'}$};
\node (D) at (1,0) {$F(\sigma(j')$};
\path[->,font=\scriptsize,>=angle 90]
(C') edge node [above] {$\tau(J'')$} (D')
(C') edge (C)
(D') edge node [right] {$F(\sigma(j'\to j''))$}(D)
(C) edge node [above] {$\tau(J')$} (D);
\end{tikzpicture}

\end{center}
is commutative.

\end{enumerate}
The face and degeneracy maps can be given by precomposition.

\begin{lem}\label{lem:Lg^* is pullback}

Let $U\xrightarrow gT$ be a morphism of $S$-schemes. Then the restriction ${}^\op\mathscr D_S|g$ is classified by $\mathscr D(T)^\op\xrightarrow{(Lg^*)^\op}\mathscr D(U)^\op$.

\end{lem}

\begin{proof}

Let's write down the opposite relative nerve of $F\colon(\Delta^1)^\op\xrightarrow{\mathscr D(T)^\op\xrightarrow{(Lg^*)^\op}\mathscr D(U)^\op}\Set_\Delta$. A simplex $\Delta^J\to(N_F\Delta^1)^\op$ is given by the following.

(1) A map $J\xrightarrow{\sigma}[1]$,

(2) and for $\emptyset\ne J'\subseteq J$,

(2a) a simplex $\Delta^{J'}\xrightarrow{\tau(J')}\mathscr D(U)^\op$ if $0\in \sigma(J')$, and

(2b) a simplex $\Delta^{J'}\xrightarrow{\tau(J')}\mathscr D(T)^\op$ if $0\notin\sigma(J')$, such that

(3) for $\emptyset\ne J''\subseteq J'\subseteq J$,

(3a) if $0\in\sigma(J'')$ or $0\notin\sigma(J')$, then we have $\tau(J'')=\tau(J')|\Delta^{J''}$, and

(3b) if $0\in\sigma(J')\setminus\sigma(J'')$, then we have $Lg^*\circ\tau(J'')=\tau(J')|\Delta^{J''}$.

For $\emptyset\ne J'\subseteq J$, let $\tau(J')=(\sigma|J',(I_i),(f_{J''})$. One can check that $(\sigma,(I_i),(f_{J'}))$ gives a $J$-simplex of ${}^\op\mathscr D^\op|g$, and that this assignment gives an isomorphism $(N_F\Delta^1)^\op\cong{}^\op\mathscr D_S|g$.

\end{proof}

\begin{prop}\label{prop:D is Cartesian}

The map ${}^\op\mathscr D_S\xrightarrow pN(\Sch_S)$ is a presentable fibration.

\end{prop}

\begin{proof}

By Lemma \ref{lem:p is inner fibration}, $p$ is an inner fibration. Moreover, by Lemma \ref{lem:Cartesian edges}, it is a Cartesian fibration. Let $U\xrightarrow gT$ be a morphism of $S$-schemes. Then by Lemma \ref{lem:Lg^* is pullback}, the pullback map over $g$ is equivalent to $Lg^*$. This functor admits a right adjoint by Proposition \ref{p:adjointFunc}. Moreover, the fibres are indeed presentable \cite{lurie2014higher}*{Proposition 1.3.5.21}. Therefore, $p$ is a presentable fibration \cite{lurie2009higher}*{Proposition 5.5.3.3}.

\end{proof}

\begin{corollary}\label{c:isQcat}
The simplicial set ${}^\op\msD_S$ is a quasi-category.
\end{corollary}

\begin{proof}
We know that $N(\Sch_S)$ is a quasi-category, indeed it is a category.
The relevant lifting problem is then solved by first lifting to schemes
then applying the proposition.
\end{proof}


\section{Descent and the stack $\SRHom$}\label{s:theorem}

\begin{proposition}\label{p:stable} Let $K$ be a small simplicial set.
Let $K\xrightarrow k\Sch_S$ be a diagram of $S$-schemes. Then
$\Gamma(k,{}^\op\mathscr D_S)$ and
$\Gamma_{\Cart}(k,{}^\op\mathscr D_S)$ are presentable stable quasi-categories.
\end{proposition}

\begin{proof}
These section categories are in fact quasi-categories via \ref{c:isQcat}.

Recall that ${}^\op\mathscr D_S\xrightarrow p N(\Sch_S)$ is a presentable fibration,
see Proposition \ref{prop:D is Cartesian}.
If follows that the category of functors $\Fun(K,{}^\op\mathscr D_S)$ is presentable
by \cite[5.5.3.6]{lurie2009higher}, for any small simplicial set $K$. It follows that
the category of sections $\Gamma(k,{}^\op\mathscr D_S)$ is presentable by \cite[5.5.3.17]{lurie2009higher}.
The subquasi-category $\Gamma_{\Cart}(k,{}^\op\mathscr D_S)\subset\Gamma(k,{}^\op\mathscr D_S)$ is an accessible localization of a presentable quasi-category \cite{lurie2009higher}*{Proposition 5.5.3.17}.

It remains to show that the categories are in fact stable.

Firstly observe that there is a $0$-section, denoted ${\bf 0}\in \Gamma(k,{}^\op\mathscr D_S)$
It sends  a simplex $\sigma=(T_i,t_ij)\in N(\Sch_S)$ to the data
$(\sigma, 0_i, (0_I)_{I\subseteq [n]})$ where $0_i$ is the zero complex on $T_i$ and
$0_I$ is the zero  morphism of degree $|I|-2$.

It follows from \ref{l:cartEdge} that this section lies inside
$\Gamma_{\Cart}(k,{}^\op\mathscr D_S)$. The zero section is a zero object by Corollary
\ref{cor:contractible fibrewise mapping spaces}.

As these section categoires are presentable they have small colimits.
By \cite[5.5.2.5]{lurie2009higher} these quasi-categories have small limits.
It follows that these section categories have fibers and cofibers. We need to show
that these two concepts agree, see \cite[Ch. 1]{lurie2014higher}.

The fibers of $p:{}^\op\msD_S\rightarrow N(\Sch_S)$ are  the usual derived
quasi-categories, see 1.3 \emph{loc. cit},  which are known to be stable.
As fibers and cofibers agree pointwise they agree in the section categories by
\cite[5.1.2.2]{lurie2009higher}.

The looping and delooping functors amount to shifts. Hence Cartesian sections
are closed under delooping and looping.

\end{proof}

If $\msX$ is a stable quasi-category then a $t$-structure on $\msX$ amounts to a
$t$-structure on its homotopy category. In other words a $t$-structure on $\msX$
amounts to two full subcategories $\msX_{\ge 0}$ and $\msX_{\le 0}$ that produce
$t$-structures on $\tau_1\msX$.

Recall in (\ref{n:truncation}) we introduced subcategories
${}^\op\msD_S^{\le 0}$ and ${}^\op\msD_S^{\ge 0}$ of ${}^\op\msD_S$. We would like
to show that taking Cartesian sections into these subcategories indeed produces
a $t$-structure on $\Gamma_{\Cart}(k,{}^\op\msD_S)$ for appropriate choice of diagram $k$.
As we do not have a good description of the homotopy category of
$\Gamma_{\Cart}(k,{}^\op\msD_S)$ we need some other tools to obtain a $t$-structure on it.

The subcategories $\msX_{\ge n}$ are localisations of $\msX$. Localisations are characterised
by sets of morphisms with respect to which we are localising, see
\cite[5.5.4.2, 5.5.4.15]{lurie2009higher}. Lets briefly recall some of these ideas.

Let $A\subseteq \msX$ be a set of morphisms (i.e edges). An object $Z\in \msX_0$ is said
to be $A$-local if for each $a:X\rightarrow Y$ in $Z$ the induced morphism
\[
\Map_\msX(Y,Z)\rightarrow \Map_\msX(X,Z)
\]
is a weak equivalence. Let $\msX'$ be the full subcategory on the $A$-local object.
Then $\msX'$ is a localisation of $\msX$, in other words the inclusion
$\msX'\hookrightarrow \msX$ has a left adjoint.

For each $S$-scheme $T$, the fiber of ${}^\op\msD_S$ over $T$ is the opposite of the usual derived
$\infty$-category which  has a $t$-structure. Let $A$ be the collection of
edges of the form $I\rightarrow \tau^{\ge 0}I$ as $I$ and $T$ vary.

Recall that ${}^\op\msD_S^{\ge 0}$ is the full subcategory of ${}^\op\msD_S$ spanned
by those complexes whose negative cohomology vanishes.

\begin{lemma}\label{l:isLocalisation}
In the notation above, a complex $K^\bullet$ is $A$-local if and only if
$K\in{}^\op\msD_S^{\ge0}$. Hence the inclusion ${}^\op\msD_S^{\ge 0}\hookrightarrow {}^\op\msD_S$
has a left adjoint.
\end{lemma}

\begin{proof}
Lets start by assuming that $K\in{}^\op\msD_S^{\ge 0}$ is complex of injectives
on a scheme $U$. We need to show $K$ is $A$-local.

In view of this, fix the truncation map
$I\rightarrow \tau^{\ge 0}I$ of a complex of injective $\mathscr O_T$-modules with quasi-coherent cohomology on a scheme $T$ and a morphism $g:U\rightarrow T$.
 Then
\begin{align*}
\Map(Lg^*\tau^{\ge 0}I,K) &\simeq \Map(\tau^{\ge 0}I, Rg_*K) \quad\mbox{see (\ref{p:adjointFunc}})\\
&\simeq \Map(I,Rg_* K) \quad\mbox{} \\
&\simeq \Map(Lg^*I,K).
\end{align*}

To complete the proof we need to show that every complex with a negative
cohomology sheaf is not $A$-local .
Suppose that $I\in\mathscr D(T)$ has $H^iI\ne0$ for some $i<0$. Let $J$ be an injective resolution of $\mathscr O_T$. Then the map
$$
H^iI=\pi_0\Map(J[i],I)\xrightarrow{\pi_0(\circ\tau^{\ge0})}\pi_0\Map(\tau^{\ge0}J[i],I)=0
$$
is not a bijection, thus $I$ is not $A$-local.
\end{proof}

We are now in a position to equip the category of Cartesian sections with a $t$-structure. As we
do not have a good handle on the homotopy category of the category of Cartesian sections, we will make
use of the following proposition from \cite{lurie2014higher}.

\begin{prop}\label{p:tStruct}
Let $\msC$ be a stable $\infty$-category. Let $i:\msC'\hookrightarrow \msC$ be a full subcategory with
localisation functor $L:\msC\rightarrow \msC'$. Then the following conditions are equivalent.
\begin{enumerate}
\item $\msC'$ is closed under extensions.
\item For each $A,B\in\msC_0$, the natural map
\[
\Ext^1(LA,B)\rightarrow \Ext^1(A,B)
\]
is injective.
\item The full subcategories $\msC^{\ge 0}=\{A|LA\simeq 0\}$ and $\msC^{\le -1}=\{A|LA\simeq A\}$
determine a $t$-structure on $\msC$.
\end{enumerate}
\end{prop}

\begin{proof}
This is proposition 1.2.1.16 of \emph{loc. cit.}
\end{proof}

\begin{prop}\label{prop:t-structure on section category}

Let $K\xrightarrow k\Sch_S$ be a diagram of $S$-schemes. Then $\Gamma_{\Cart}(k,{}^\op\mathscr D_S)$ is a presentable stable category. Furthermore  $(\Gamma_{\Cart}(k,{}^\op\mathscr D_S^{\ge0}),\Gamma_{\Cart}(k,{}^\op\mathscr D_S^{\le0}))$ is an accessible left complete t-structure.

\end{prop}

\begin{proof}

First note that $\Gamma(k,{}^\op\mathscr D_S)\supseteq\Gamma_{\Cart}(k,{}^\op\mathscr D_S)$
is a localisation by \cite[5.5.3.17]{lurie2009higher}. We claim that the localization map $\Gamma(k,{}^\op\mathscr D_S)\to\Gamma_{\Cart}(k,{}^\op\mathscr D_S)$ is left exact. It will be enough to show that $\Gamma_{\Cart}(k,{}^\op\mathscr D_S)\subset\Gamma(k,{}^\op\mathscr D_S)$ is closed under delooping \cite{lurie2014higher}*{Lemma 1.1.3.3 and Proposition 1.4.4.9}. But as we have seen in \ref{l:cartEdge}, an edge in ${}^\op\mathscr D_S$ is Cartesian precisely when if it is given by a quasi-isomorphism $Lg^*I\to J$. This is stable under translation, which proves the claim.

 Given $I\in \Gamma_{\Cart}(k,{}^\op\mathscr D_S^{\ge0})$ there is an edge
 $I\rightarrow \tau^{\ge 0}I$
 which we can complete to a fiber sequence
 $$
I'\to I\xrightarrow{\tau^{\ge0}}I''.
$$
as the Cartesian section category is stable, (\ref{p:stable}). Now by Corollary \ref{cor:contractible fibrewise mapping spaces}, $\Map_{\Gamma_{\Cart}(k,{}^\op\mathscr D_S)}(I',I'')$ is contractible. The second criterion of (\ref{p:tStruct}) is now verified. We thus have a $t$-structure.

To show left completeness, it is enough to show that $\cap_n\Gamma_{\Cart}(k,{}^\op\mathscr D_S^{\ge n})$ only has zero objects. It only contains Cartesian diagrams of cohomologically trivial sheaves, thus pointwise zero objects, which proves the claim.

\end{proof}

Given a stable infinity category with a $t$-structure $(\msD,\msD^{\ge 0}, \msD^{\le 0})$ we can form
its heart, $\msD^\heartsuit = \msD^{\ge 0}\cap \msD^{\le 0}$. It is equivalent to the nerve of an ordinary abelian
category.  In the case where our diagram $k$ is in fact the Cech nerve of a
cover $U\rightarrow T$, so $k:N(U^\bullet/T)\rightarrow \Sch_S$, it is in fact the category of
descent data for a quasi-coherent sheaf on $T$. Hence it is equivalent to the category of
quasi-coherent sheaves on $T$, denoted $\QCoh(T)$.

In the situation where the heart is Grothendieck abelian category, we can now ask if the original stable quasi-category
is the derived quasi-category of this abelian category. This is answered by the following result:

\begin{proposition}\label{prop:bounded below universal}
Let  $(\msD,\msD^{\ge 0}, \msD^{\le 0})$ be a stable $\infty$-category with a right complete $t$-structure.
Suppose further that $\msD^\heartsuit$ is a Grothendieck abelian category.
Then there is a functor
$F:\msD^+(\msD^\heartsuit)\rightarrow\msD$ extending the inclusion of the heart inside $\msD$, which is unique up to a contractible space of equivalences.
Moreover, the following are equivalent.
\begin{enumerate}
\item The functor $F$ is fully faithful.
\item For every pair of objects $X,I\in \msD^\heartsuit$, if $I$ is injective, then
$\Ext^i_\msD(X,I)=0$ for $i>0$.
\end{enumerate}
Further, if these conditions are satisfied the essential image of $F$ is the full subcategory
$\cup_{n\in\mathbf Z} \msD_{\ge -n}$.
\end{proposition}

\begin{proof}
This is \cite[1.3.3.7]{lurie2014higher}.
\end{proof}

For a cosimplicial abelian group $A^\bullet$, we let $\pi^s(A)=H^s(A,\sum(-1)^id^i)$.

\begin{thm}\label{thm:descent for complexes}

The Cartesian fibration ${}^\op\mathscr D^+_S\to\Sch_S$ satisfies fppf descent.

\end{thm}

\begin{proof}

Let $U\xrightarrow gT$ be an fppf covering over $S$. Let $N(\Delta^\op)\xrightarrow{k}\Sch_S$ denote its \v Cech nerve. Let $\mathscr C=\Gamma_{\Cart}(k,{}^\op\mathscr D_S)$. By Proposition \ref{prop:t-structure on section category}, we can apply Proposition \ref{prop:bounded below universal}. By descent for quasi-coherent sheaves, we have $(\mathscr C^\op)^\heartsuit\simeq\QCoh(T)$. By applying Proposition \ref{prop:bounded below universal} to $\mathscr C^\op$, it is enough to show that for $M,I\in\QCoh(T)$ with $I$ injective, we have $\Ext^i_{\mathscr C^\op}(M,I)=0$ for $i>0$. By Proposition \ref{prop:Map of Gamma_Cart}, $\Map_{\mathscr C^\op}(M,I)$ is the total complex of a cosimplicial space $m\mapsto\Map_{U_m}(M_{U_m},I_{U_m})$. Therefore, we can apply the Bousfield--Kan spectral sequence to calculate $\Ext^i_{\mathscr C^\op}(M,I)=\pi_0\Map_{\mathscr C^\op}(M,I[i])$. We have \cite{goerss1999simplicial}*{VIII, Proposition 1.15}
$$
E_2^{s,t}=\pi^s\pi_t\Map_{U_\bullet}(M_{U_\bullet},I_{U_\bullet}[i])=\pi^s\Ext^0_{U_\bullet}(M_{U_\bullet},I_{U_\bullet}[i-t]).
$$
We can see that these terms vanish for $i\ne t$. Therefore, we have
$$
E_2^{s,t}=\begin{cases}
\pi^s\Ext^0_{U_\bullet}(M_{U_\bullet},I_{U_\bullet}) & t=i\\
0 & t\ne i.
\end{cases}
$$
But we have $E^{s,t}_2\Rightarrow\pi_{t-s}\mathrm{Tot}\Map_{U_\bullet}(M_{U_\bullet},I_{U_\bullet}[i])$. Therefore, $\Ext^i_{\mathscr C^\op}(M,I)=\pi_0\mathrm{Tot}\Map_{U_\bullet}(M_{U_\bullet},I_{U_\bullet}[i])=0$ as we wanted.

\end{proof}

\begin{cor}

For any $S$-scheme $T$, we have a natural equivalence $\mathscr D^+(T)\to\QC^+(T)$ \cite{ben-zvi2010integral}*{\S3.1}

\end{cor}

We refer the reader to \emph{loc. cit.} for a definition of $\QC^+(T)$.

\begin{cons}

Let $\mathscr X\to K$ be a Cartesian fibration. Suppose that $K$ is a quasi-category with a final object $S\in K$. Let $I,J\in\mathscr X(S)$. Then by Lemma \ref{lem:Cartesian section}, there exists a Cartesian section $K\xrightarrow{I_\bullet}\mathscr X$ with $I_S=I$. We can then form the \emph{mapping prestack} as the left vertical arrow in the pullback diagram
\begin{center}

\begin{tikzpicture}[xscale=2,yscale=1.5]
\node (C') at (0,1) {$\SMap_{\mathscr X}(I,J)$};
\node (D') at (1,1) {$\mathscr X^{/J}$};
\node (C) at (0,0) {$K$};
\node (D) at (1,0) {$\mathscr X$.};
\path[->,font=\scriptsize,>=angle 90]
(C') edge (D')
(C') edge (C)
(D') edge (D)
(C) edge node [above] {$I_\bullet$} (D);
\end{tikzpicture}

\end{center}
Note that as $\SMap_{\mathscr X}(I,J)\to K$ is a right fibration, it can be straightened to the presheaf
$$
T\mapsto\Map_{\mathscr X}(I_T,J)\simeq\Map_{\mathscr X(T)}(I_T,J_T).
$$
In the case of the Cartesian fibration ${}^\op\mathscr D_S\to\Sch_S$, we write ${}^\op\SRHom_S(J,I)=\SMap_{{}^\op\mathscr D_S}(I,J)$. By Theorem \ref{thm:descent for complexes}, it is a stack.

\end{cons}

\begin{rem}

Note that by Lemma \ref{lem:Cartesian section}, the restrict to $S$ map $\Gamma_{\Cart}(K,\mathscr X)\to\mathscr X(S)$ is a trivial fibration. Therefore, $\SMap_{\mathscr X}(I,J)$ does not depend on the choice of $I_\bullet$ up to equivalence.

\end{rem}

\section{An explicit construction of loop groups in $\infty$-topoi}

Let $E\in\mathscr D(S)^+$ be a bounded below complex of $\mathscr O_S$-modules. In this subsection, we give a description of the loop group $\Omega(E,(\mathscr D_S^+)^\simeq)$. We will give a general construction for a pointed object $\ast\xrightarrow xX$ in an $\infty$-topos $\mathscr X$.

By definition, an $\infty$-topos is a left exact localization of a presheaf category $\mathscr P(\mathscr C)$ of a small quasi-category $\mathscr C$ \cite{lurie2009higher}*{\S6.1}. We will use the following equivalent description \cite{lurie2009higher}*{Proposition 5.1.1.1}. The simplicial overcategory $(\Set_\Delta)_{/\mathscr C}$ can be equipped with the contravariant model structure, the fibrant objects of which are exactly the right fibrations \cite{lurie2009higher}*{Corollary 2.2.3.12}. Then the presheaf quasi-category $\mathscr P(\mathscr C)=\Fun(\mathscr C^\op,\mathscr S)$ is equivalent to the categorical nerve $\mathscr P'(\mathscr C)$ of the full simplicial subcategory of $(\Set_\Delta)_{/\mathscr C}$ on fibrant objects. Therefore, we will represent the objects in the $\infty$-topos $\mathscr X$ as right fibrations $X\to\mathscr C$. In this description, a pointed object is a section $\mathscr C\xrightarrow xX$.

The loop group $\Omega(x,\mathscr X)$ is the \emph{\v Cech nerve} of $\mathscr C\xrightarrow xX$ \cite{lurie2009higher}*{below Proposition 6.1.2.11}. Let's recall this notion. A \emph{groupoid object} in $\mathscr X$ is a simplicial object $\Delta^\op\xrightarrow G\mathscr X$ such that for all $n\ge0$ and all partitions $[n]=S\cup S'$ such that $S\cap S'=\{s\}$, the square
\begin{center}

\begin{tikzpicture}[xscale=2,yscale=1.5]
\node (C') at (0,1) {$G_n$};
\node (D') at (1,1) {$G_{S'}$};
\node (C) at (0,0) {$G_S$};
\node (D) at (1,0) {$G_{\{s\}}$};
\path[->,font=\scriptsize,>=angle 90]
(C') edge (D')
(C') edge (C)
(D') edge (D)
(C) edge (D);
\end{tikzpicture}

\end{center}
is homotopy Cartesian. Let $x\xrightarrow fy$ be a morphism in $\mathscr X$. Then its \v Cech nerve is a groupoid object $G$ such that there exists an augmented simplicial object $\Delta_+^\op\xrightarrow{\Tilde G}\mathscr X$, such that

\begin{enumerate}

\item $\Tilde G|\Delta^\op=G$, and

\item the square
\begin{center}

\begin{tikzpicture}[xscale=2,yscale=1.5]
\node (C') at (0,1) {$\Tilde G_1$};
\node (D') at (1,1) {$\Tilde G_{\{1\}}$};
\node (C) at (0,0) {$\Tilde G_{\{0\}}$};
\node (D) at (1,0) {$\Tilde G_{-1}$};
\path[->,font=\scriptsize,>=angle 90]
(C') edge (D')
(C') edge (C)
(D') edge (D)
(C) edge (D);
\end{tikzpicture}

\end{center}
is homotopy Cartesian.

\end{enumerate}

As this involves a lot of homotopy fibre products, the loop group structure is highly inexplicit. Our construction makes the composition, associativity, etc.~diagrams explicit in one bisimplicial set.

\begin{cons}

Let $G$ be the following simplicial object of simplicial sets over $\mathscr C$. Let $G_0=\mathscr C$. If $m>0$, then the $(m,n)$-simplices are the maps $\Delta^m\times\Delta^n\to X$ such that their restriction to $\sk_0\Delta^m\times\Delta^n$ factors through $x$. Both the horizontal and vertical face and degeneracy maps can be given by restriction. Unless we state otherwise, we will think of $G$ as a vertical bisimplicial set, that is as a simplicial object $\Delta^\op\xrightarrow{m\mapsto G_m}\mathscr X$.

\end{cons}

\begin{thm}

The bisimplicial set $G$ is a simplicial right fibration over $\mathscr C$, and it is the \v Cech nerve of the classifying map $\mathscr C\xrightarrow xX$.

\end{thm}

\begin{proof}

Since $\mathscr P'(\mathscr C)\to\mathscr X$ is a left exact localization, that is it commutes with finite limits, it is enough to show that $G$ is the \v Cech nerve of $x$ in $\mathscr P'(\mathscr C)$. First of all, let's prove that the $G_m$ are right fibrations. Since $G_0=\mathscr C$, let's assume $m>0$. Then we need to show that for all $0<k\le m$, any lifting problem
\begin{center}

\begin{tikzpicture}[xscale=2,yscale=1.5]
\node (C') at (0,1) {$\Lambda^n_k$};
\node (D') at (1,1) {$G_m$};
\node (C) at (0,0) {$\Delta^n$};
\node (D) at (1,0) {$\mathscr C$};
\path[->,font=\scriptsize,>=angle 90]
(C') edge node [above] {$\sigma$} (D')
(C') edge (C)
(D') edge (D)
(C) edge [dashed] (D')
(C) edge (D);
\end{tikzpicture}

\end{center}
has a solution. Since $\sigma|(\sk_0\Delta^m\times\Delta^n)$ needs to factor through $x$, we don't have to worry about the projection to $\mathscr C$. Therefore, it is enough to show that the inclusion $(\sk_0\Delta^m\times\Delta^n)\cup(\Delta^m\times\Lambda^n_k)\subset\Delta^m\times\Delta^n$ is right anodyne. That in turn holds by \cite{lurie2009higher}*{Corollary 2.1.2.7}.

We next claim that the square
\begin{center}

\begin{tikzpicture}[xscale=2,yscale=1.5]
\node (C') at (0,1) {$G_1$};
\node (D') at (1,1) {$\mathscr C$};
\node (C) at (0,0) {$\mathscr C$};
\node (D) at (1,0) {$X$};
\path[->,font=\scriptsize,>=angle 90]
(C') edge (D')
(C') edge (C)
(D') edge node [right] {$x$} (D)
(C) edge node [above] {$x$} (D);
\end{tikzpicture}

\end{center}
is homotopy Cartesian. The square
\begin{center}

\begin{tikzpicture}[xscale=2,yscale=1.5]
\node (C') at (0,1) {$G_1$};
\node (D') at (1,1) {$X^{/x}$};
\node (C) at (0,0) {$\mathscr C$};
\node (D) at (1,0) {$X$};
\path[->,font=\scriptsize,>=angle 90]
(C') edge (D')
(C') edge (C)
(D') edge node [right] {$\res$} (D)
(C) edge node [above] {$x$} (D);
\end{tikzpicture}

\end{center}
is strict Cartesian, and as the restriction map $X^{/x}\to X$ is a right fibration, it is morever homotopy Cartesian. Therefore, it is enough to show that the projection $X^{/x}\to\mathscr C$ is a trivial fibration, since that will imply that $X^{/x}\to X$ is a right fibrant resolution of $\mathscr C\xrightarrow xX$, and therefore $G_1=\mathscr C\times^h_X\mathscr C$. The projection $X^{/x}\to\mathscr C$ is a right fibration, since it's a composite of such \cite{lurie2009higher}*{Proposition 4.2.1.6}, therefore it is enough to show that it's a contravariant equivalence. That can be checked fibrewise \cite{lurie2009higher}*{Corollary 2.2.3.13}, that is, assuming that $X$ is a Kan complex and $x\in X$, we need to show that $X^{/x}$ is contractible. The restriction map $X^{/x}\to X$ is a right fibration \cite{lurie2009higher}*{Proposition 4.2.1.6}, therefore as $X$ is a Kan complex, $X^{/x}$ is one too. We have a categorical equivalence $X_{/x}\to X^{/x}$ \cite{lurie2009higher}*{Proposition 4.2.1.5} between Kan complexes, therefore it is a weak homotopy equivalence \cite{lurie2009higher}*{Lemma 3.1.3.2}. Finally, $X^{/x}$ is contractible: for all $n\ge0$, a lifting problem
\begin{center}

\begin{tikzpicture}[xscale=2,yscale=1.5]
\node (C') at (0,1) {$\partial\Delta^n$};
\node (D') at (1,1) {$X_{/x}$};
\node (C) at (0,0) {$\Delta^n$};
\node (D) at (1,0) {$\ast$};
\path[->,font=\scriptsize,>=angle 90]
(C') edge (D')
(C') edge (C)
(D') edge (D)
(C) edge [dashed] (D')
(C) edge (D);
\end{tikzpicture}

\end{center}
is a lifting problem
\begin{center}

\begin{tikzpicture}[xscale=2,yscale=1.5]
\node (C') at (0,1) {$\Lambda^{n+1}_{n+1}$};
\node (D') at (1,1) {$X$};
\node (C) at (0,0) {$\Delta^{n+1}$};
\node (D) at (1,0) {$\ast,$};
\path[->,font=\scriptsize,>=angle 90]
(C') edge (D')
(C') edge (C)
(D') edge (D)
(C) edge [dashed] (D')
(C) edge (D);
\end{tikzpicture}

\end{center}
therefore it has a solution.

Let us finally show that $G$ is a groupoid object of right fibrations over $\mathscr C$. That is, we need to show that for all $m\ge2$ and all partitions $I\cup J=[m]$ such that $I\cap J=\{i\}$, the diagram
\begin{center}

\begin{tikzpicture}[xscale=2,yscale=1.5]
\node (C') at (0,1) {$G_m$};
\node (D') at (1,1) {$G_J$};
\node (C) at (0,0) {$G_I$};
\node (D) at (1,0) {$G_{\{i\}}$};
\path[->,font=\scriptsize,>=angle 90]
(C') edge (D')
(C') edge (C)
(D') edge (D)
(C) edge (D);
\end{tikzpicture}

\end{center}
is homotopy Cartesian, or equivalently, the canonical map $G_m\xrightarrow fG_I\times_{G_{\{i\}}}^hG_J$ is a contravariant equivalence. Since $G_J\to G_{\{i\}}=\mathscr C$ is a right fibration, we have $G_I\times_{G_{\{i\}}}^hG_J=G_I\times_{G_{\{i\}}}G_J$. Therefore, it is enough to show that for all $c\in\mathscr C$, the fibre $f_c$ is a weak homotopy equivalence \cite{lurie2009higher}*{Corollary 2.2.3.13}. That is, we can assume that $\mathscr C=\{c\}$. We claim that in this case $f$ is a trivial fibration. That is, we need to show that for every $n\ge1$, every lifting problem
\begin{center}

\begin{tikzpicture}[xscale=2,yscale=1.5]
\node (C') at (0,1) {$\partial\Delta^n$};
\node (D') at (1,1) {$G_m$};
\node (C) at (0,0) {$\Delta^n$};
\node (D) at (1,0) {$G_I\times_{G_{\{i\}}}G_J$};
\path[->,font=\scriptsize,>=angle 90]
(C') edge (D')
(C') edge (C)
(D') edge (D)
(C) edge [dashed] (D')
(C) edge (D);
\end{tikzpicture}

\end{center}
has a solution. Unwinding the construction, it is enough to show that the inclusion
$$
(\sk_0\Delta^m\times\Delta^n)\cup((\Delta^I\cup\Delta^J)\times\Delta^n)\cup(\Delta^m\times\partial\Delta^n)\subset\Delta^m\times\Delta^n
$$
is anodyne. Since $I\cup J=[m]$, we have $\sk_0\Delta^m\subseteq\Delta^I\cup\Delta^J$, therefore we can drop the first component of the union. Then we are done by Lemma \ref{lem:partition is anodyne} \cite{goerss1999simplicial}*{Corollary I.4.6}.

\end{proof}

\begin{lem}\label{lem:partition is anodyne}

Let $m\ge2$ and consider a partition $I\cup J=[m]$ such that $I\cap J\ne\emptyset$. Then the inclusion $\Delta^I\cup\Delta^J\subseteq\Delta^m$ is anodyne.

\end{lem}

\begin{proof}

We can assume neither $I$ or $J$ is $[m]$. Let us use induction on $k=|(I\cup J)\setminus(I\cap J)|\ge2$. The $k=2$ case is proven by Lemma \ref{lem:facets is anodyne}. Suppose that there are $i\ne j\in I\setminus J$. Then as we have a pushforward diagram
\begin{center}

\begin{tikzpicture}[xscale=3,yscale=1.5]
\node (C') at (0,1) {$\Delta^I\cup\Delta^{J\cup\{i\}}$};
\node (D') at (1,1) {$\Delta^{I\setminus\{j\}}\cup\Delta^{J\cup\{i\}}$};
\node (C) at (0,0) {$\Delta^I\cup\Delta^J$};
\node (D) at (1,0) {$\Delta^{I\setminus\{j\}}\cup\Delta^J,$};
\node at (.5,.5) {$\lrcorner$};
\path[->,font=\scriptsize,>=angle 90]
(D') edge (C')
(C) edge (C')
(D) edge (D')
(D) edge (C);
\end{tikzpicture}

\end{center}
the statement is proven by the induction hypothesis \cite{goerss1999simplicial}*{Proposition 4.2}. The $i\ne j\in J\setminus I$ case can be proven the same way.

\end{proof}

\begin{lem}\label{lem:facets is anodyne}

Let $S\xrightarrow i\Delta^n$ be the inclusion of the union of a proper subset of the set of facets of $\Delta^n$. Then $i$ is anodyne.

\end{lem}

\begin{proof}

Let us use induction on $n\ge1$, the $n=1$ case being satisfied by definition. Suppose $n>1$. Let us use induction on the number $m\ge1$ of facets missing from $S$. If $m=1$, then $S$ is a horn, so $i$ is anodyne by definition. Suppose that $m>1$, and that the facet $\Delta^{[n]\setminus\{i\}}$ is missing from $S$. By the induction hypothesis on $n$, the inclusion $S\cap\Delta^{[n]\setminus\{i\}}\subset\Delta^{[n]\setminus\{i\}}$ is anodyne, thus so is its pushout $S\subset S\cup \Delta^{[n]\setminus\{i\}}$. By the induction hypothesis on $m$, this in turn shows that $i$ is the composite of anodyne maps $S\subset S\cup \Delta^{[n]\setminus\{i\}}\subset\Delta^n$.

\end{proof}

\end{document}